\pdfoutput=1
\documentclass[11pt]{article}
\usepackage{authblk}
\usepackage{extarrows}
\usepackage[toc,page]{appendix}
\usepackage[top=2cm, bottom=2cm, left=2cm, right=2cm]{geometry}

\usepackage{color}
\usepackage{helvet}         
\usepackage{courier}        
\usepackage{type1cm}        
 
\usepackage{framed}
\usepackage{tikz}
\usepackage{makeidx}         
\usepackage{graphicx}        
\usepackage{multicol}        
\usepackage[bottom]{footmisc}

\usepackage{amsmath}
\usepackage{amssymb}
\usepackage{bbold}
\usepackage{amsthm}
\usepackage{subcaption}
\usepackage{sidecap}
\usepackage{floatrow}
\usepackage{pdflscape}
\usepackage{comment}
\usepackage[font=small]{caption}
\usepackage{enumitem}
\usepackage{esint}

\usepackage{scalerel}

\newtheorem{theorem}{Theorem}
\newtheorem{corollary}[theorem]{Corollary}
\newtheorem{lemma}[theorem]{Lemma}

\newtheorem{proposition}[theorem]{Proposition}

\theoremstyle{definition}
\newtheorem{definition}[theorem]{Definition}
\theoremstyle{remark}

\numberwithin{theorem}{section}
\numberwithin{figure}{section}
\numberwithin{equation}{section}

\begin{document}

\title{Multiple Ising interfaces in annulus and $2N$-sided radial SLE}
\bigskip{}
\author[1]{Yu Feng\thanks{yufeng\_proba@163.com}}
\author[1]{Hao Wu\thanks{hao.wu.proba@gmail.com. Supported by Beijing Natural Science Foundation (JQ20001).}}
\author[1]{Lu Yang\thanks{luyang\_proba@126.com}}
\affil[1]{Tsinghua University, China}

\date{}

%
%

\global\long\def\dist{\mathrm{dist}}
\global\long\def\SLE{\mathrm{SLE}}
\global\long\def\PPrSLE3{\PP_{\mathrm{rSLE(3)}}}
\global\long\def\PPcSLE3{\PP_{\mathrm{cSLE(3)}}}
\global\long\def\PPIsing{\PP_{\mathrm{Ising}}}
\global\long\def\diam{\mathrm{diam}}
\global\long\def\free{\mathrm{free}}
\global\long\def\edge#1#2{\langle #1,#2 \rangle}
\global\long\def\ns{\mathrm{ns}}
\global\long\def\total{\mathrm{Total}}
\global\long\def\Ising{\mathrm{Ising}}
\global\long\def\thetaf{\vartheta}
\global\long\def\eps{\epsilon}
\global\long\def\doublemap{\tilde{\omega}}
\global\long\def\doubledom{\tilde{\Omega}}
\global\long\def\mQ{\mathbb{Q}_{\mathrm{m}}}
\global\long\def\mP{\mathbb{P}_{\mathrm{m}}}
\global\long\def\mE{\mathbb{E}_{\mathrm{m}}}
\global\long\def\conf{\mathrm{Conf}}
\global\long\def\good{\mathrm{E}}
\global\long\def\totalH{\mathcal{Z}}
\global\long\def\totalD{\mathcal{G}_{\mathrm{Total}}}
\global\long\def\U{\mathbb{U}}
\global\long\def\partiIsing{\mathcal{H}}
\global\long\def\confdis{\mathrm{d}^{(a,2N)}_{\mathrm{conf}}}
\global\long\def\bcint{\mathrm{B}_{\mathrm{int}}}
\global\long\def\bcintone{\mathrm{B}}
\global\long\def\PIsing{\mathbb{P}}
\global\long\def\EIsing{\mathbb{E}}
\global\long\def\Z{\mathbb{Z}}
\global\long\def\T{\mathbb{T}}
\global\long\def\HH{\mathbb{H}}
\global\long\def\LA{\mathcal{A}}
\global\long\def\LB{\mathcal{B}}
\global\long\def\LC{\mathcal{C}}
\global\long\def\LD{\mathcal{D}}
\global\long\def\LF{\mathcal{F}}
\global\long\def\LK{\mathcal{K}}
\global\long\def\LE{\mathcal{E}}
\global\long\def\LG{\mathcal{G}}
\global\long\def\LI{\mathcal{I}}
\global\long\def\LJ{\mathcal{J}}
\global\long\def\LL{\mathcal{L}}
\global\long\def\LM{\mathcal{M}}
\global\long\def\LN{\mathcal{N}}
\global\long\def\LQ{\mathcal{Q}}
\global\long\def\LR{\mathcal{R}}
\global\long\def\LT{\mathcal{T}}
\global\long\def\LS{\mathcal{S}}
\global\long\def\LU{\mathcal{U}}
\global\long\def\LV{\mathcal{V}}
\global\long\def\LW{\mathcal{W}}
\global\long\def\LX{\mathcal{X}}
\global\long\def\LY{\mathcal{Y}}
\global\long\def\PartF{\mathcal{Z}}
\global\long\def\LH{\mathcal{H}}
\global\long\def\R{\mathbb{R}}
\global\long\def\C{\mathbb{C}}
\global\long\def\N{\mathbb{N}}
\global\long\def\Z{\mathbb{Z}}
\global\long\def\E{\mathbb{E}}
\global\long\def\PP{\mathbb{P}}
\global\long\def\QQ{\mathbb{Q}}
\global\long\def\A{\mathbb{A}}
\global\long\def\one{\mathbb{1}}
\global\long\def\chamber{\mathfrak{X}}
\global\long\def\metric{\mathrm{Dist}}
\global\long\def\cond{\,|\,}
\global\long\def\la{\langle}
\global\long\def\ra{\rangle}
\global\long\def\outb{\partial_{\mathrm{o}}}
\global\long\def\intb{\partial_{\mathrm{i}}}
\global\long\def\Riem{\mathbb{C}_{\infty}}
\global\long\def\Im{\operatorname{Im}}
\global\long\def\Re{\operatorname{Re}}
\global\long\def\Pfree{\mathbb{P}_{\mathrm{Free}}}
\global\long\def\Efree{\mathbb{E}_{\mathrm{Free}}}

\global\long\def\ud{\mathrm{d}}

\global\long\def\ii{\mathfrak{i}}
\global\long\def\rr{\mathfrak{r}}
\global\long\def\cc{\mathfrak{c}}

\newcommand{\armexp}{A_{2N}}
\newcommand{\cvgexp}{B_{2N}}

\newcommand{\unitD}{\mathbb{D}}
\global\long\def\LP{\mathrm{LP}}

\global\long\def\bs{\boldsymbol}
\maketitle
\vspace{-1cm}
\begin{center}
\begin{minipage}{0.95\textwidth}
\abstract{We consider critical planar Ising model in annulus with alternating boundary conditions on the outer boundary and free boundary conditions in the inner boundary. As the size of the inner hole goes to zero, the event that all interfaces get close to the inner hole before they meet each other is a rare event. We prove that the law of the collection of the interfaces conditional on this rare event converges in total variation distance to the so-called $2N$-sided radial SLE$_3$, introduced by~\cite{HealeyLawlerNSidedRadialSLE}. The proof relies crucially on an estimate for multiple chordal SLE. Suppose $(\gamma_1, \ldots, \gamma_N)$ is chordal $N$-SLE$_{\kappa}$ with $\kappa\in (0,4]$ in the unit disc, and we consider the probability that all $N$ curves get close to the origin. We prove that the limit $\lim_{r\to 0+}r^{-A_{2N}}\PP[\dist(0,\gamma_j)<r, 1\le j\le N]$ exists, where $A_{2N}$ is the so-called $2N$-arm exponents and $\dist$ is Euclidean distance. We call the limit Green's function for chordal $N$-SLE$_{\kappa}$. This estimate is a generalization of previous conclusions with $N=1$ and $N=2$ proved in~\cite{LawlerRezaeiNaturalParameterization, LawlerRezaeiMinkowskiContent} and~\cite{ZhanGreen2SLE} respectively.}

\end{minipage}
\end{center}

\tableofcontents

\newpage
\section{Introduction}
\label{sec::intro}
The Ising model is one of the most studied lattice model in statistical physics. 
Due to recent celebrated work of Chelkak and Smirnov~\cite{ChelkakSmirnovIsing}, it is proved that, at the critical temperature, the planar Ising interface is conformally invariant in the scaling limit. 
Based on the conformal invariance, one is able to identify the scaling limit of Ising interface in simply connected domains as Schramm Loewner evolution $\SLE_3$~\cite{CDCHKSConvergenceIsingSLE}. 
In this article, we consider critical planar Ising model in annulus.

In~\cite{IzyurovIsingMultiplyConnectedDomains}, Izyurov systematically studied the scaling limit of Ising interface in finitely connected domains. 
In our article, we study the Ising interface in annulus. Our setup is similar to that of~\cite{IzyurovIsingMultiplyConnectedDomains}, but we focus on a different aspect of the Ising interfaces. We  denote standard annulus as follows: 
\begin{align*}
\A_p:=\{z\in \C:e^{-p}<|z|<1\},\quad \text{for }p>0. 
\end{align*}
Suppose $x_1, \ldots, x_{2N}$ are distinct boundary points lying on the unit circle $\partial\unitD$ in counterclockwise order. Suppose $(\LA^{\delta}; x_1^{\delta}, \ldots, x_{2N}^{\delta})$ is an approximation of $(\A_p; x_1, \ldots, x_{2N})$ and we consider critical Ising model in $\LA^{\delta}$ with the following boundary conditions: $\oplus$ on the boundary arc $(x_{2j-1}^{\delta}x_{2j}^{\delta})$ and $\ominus$ on the boundary arc $(x_{2j}^{\delta}x_{2j+1}^{\delta})$ for $j\in\{1, \ldots, N\}$ and free in the inner hole, see precise definition in Proposition~\ref{prop::joint_Ising_annulus}. We consider the Ising interfaces $\eta_1^{\delta}, \ldots, \eta_{2N}^{\delta}$ starting from $x_1^{\delta}, \ldots, x_{2N}^{\delta}$ respectively. 
Using techniques from~\cite{IzyurovIsingMultiplyConnectedDomains} and~\cite{CHI22}, one is able to derive that the joint distribution of $(\eta_1^{\delta}, \ldots, \eta_{2N}^{\delta})$ has a scaling limit. Suppose $(\eta_1, \ldots, \eta_{2N})$ is the scaling limit, these interfaces may meet each other or hit the inner hole, see Figure~\ref{fig::IsingAnnulus}. We are interested in the case when all $2N$ interfaces hit the inner hole. As $p\to \infty$, the event that all $2N$ interfaces hit the inner hole before they meet each other is a  rare event and has small probability. We will derive the asymptotic of the probability of this rare event as $p\to \infty$, see~\eqref{eqn::IsingAsy}. Then, we consider the law of $(\eta_1, \ldots, \eta_{2N})$ conditional on this rare event, our main conclusion is that, this conditional law converges to the so-called $2N$-sided radial $\SLE_3$ introduced by~\cite{HealeyLawlerNSidedRadialSLE} as $p\to\infty$, see Theorem~\ref{thm::Ising_annulus_radialSLE}. 

\begin{figure}[ht!]
\begin{center}
\begin{subfigure}[b]{0.13\textwidth}
\begin{center}
\includegraphics[width=\textwidth]{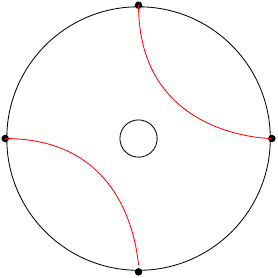}
\end{center}
\caption{}
\end{subfigure}
\begin{subfigure}[b]{0.13\textwidth}
\begin{center}
\includegraphics[width=\textwidth]{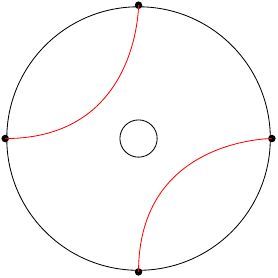}
\end{center}
\caption{}
\end{subfigure}
\begin{subfigure}[b]{0.13\textwidth}
\begin{center}
\includegraphics[width=\textwidth]{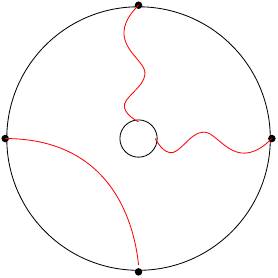}
\end{center}
\caption{}
\end{subfigure}
\begin{subfigure}[b]{0.13\textwidth}
\begin{center}
\includegraphics[width=\textwidth]{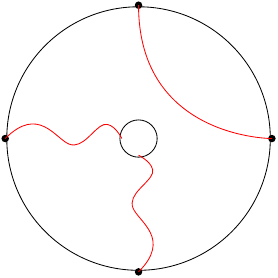}
\end{center}
\caption{}
\end{subfigure}
\begin{subfigure}[b]{0.13\textwidth}
\begin{center}
\includegraphics[width=\textwidth]{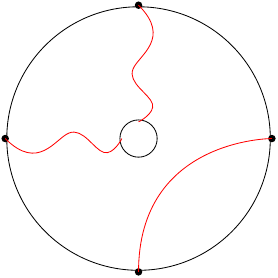}
\end{center}
\caption{}
\end{subfigure}
\begin{subfigure}[b]{0.13\textwidth}
\begin{center}
\includegraphics[width=\textwidth]{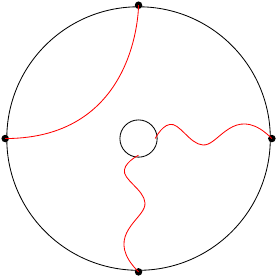}
\end{center}
\caption{}
\end{subfigure}
\begin{subfigure}[b]{0.13\textwidth}
\begin{center}
\includegraphics[width=\textwidth]{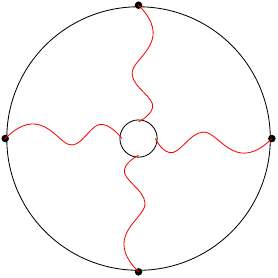}
\end{center}
\caption{}
\end{subfigure}
\end{center}
\caption{\label{fig::IsingAnnulus}
Consider critical Ising interfaces in annular polygon $(\mathbb{A}_p; x_1, x_2, x_3, x_4)$ with boundary conditions: $\oplus$ on $(x_1x_2)\cup(x_3x_4)$ and $\ominus$ on $(x_2x_3)\cup(x_4x_1)$ and free on the inner hole. There are four interfaces starting from $x_1, x_2, x_3, x_4$ respectively. 
There are seven possibilities for the connectivity of interfaces: in~(a) and~(b), four interfaces meet each other; in~(c)-(f), two interfaces meet each other and the other two interfaces connect the marked points to the inner hole; in~(g), all four interfaces hit the inner hole before they meet each other. In this article, we focus on the scenario in (g) and show that, conditional on this rare event, the law of the collection of the four interfaces converges to 4-sided radial $\SLE_3$.}
\end{figure}

\subsection{Convergence of Ising interfaces to radial SLE} 
\label{subsec::Ising_intro}

We say that a bounded planar domain $\mathcal{A}$ is an \textit{annular domain} if the following 3 conditions are satisfied: (1): the complement of $\mathcal{A}$ in the Riemann sphere $\C\cup\{\infty\}$ has two connected components and both of them contain more than one point; (2): we denote by $D_{\mathrm{i}}$ the bounded connected component of $\C\setminus \mathcal{A}$; the boundary of $D_{\mathrm{i}}$ is a smooth curve and we denote $\intb:=\partial D_{\mathrm{i}}$ and call it the \textit{inner boundary} of $\mathcal{A}$; (3): the boundary of the simply connected domain $\mathcal{A}\cup D_{\mathrm{i}}$ is locally connected, and we denote $\outb:=\partial(\mathcal{A}\cup D_{\mathrm{i}})$ and call it the \textit{outer boundary} of $\mathcal{A}$. 
For $N\ge 1$, we say $(\mathcal{A};x_1,\ldots,x_{2N})$ is an \textit{annular polygon} if $\mathcal{A}$ is an annular domain and $x_1,\ldots x_{2N}$ are distinct  points lying on $\outb$ in counterclockwise order. We denote by $(x_1x_2)$ the boundary arc from $x_1$ to $x_2$ in counterclockwise order along $\outb$.
We usually write $x_{2N+1}=x_1$ and $x_0=x_{2N}$ by convention.
In this article, we consider Ising model in annular domains.

\paragraph*{Ising model.}
For definiteness, we consider subgraphs of the square lattice $\Z^2$, which is the graph with vertex set $V(\Z^2):=\{ z = (m, n) \colon m, n\in \Z\}$ and edge set $E(\Z^2)$ given by edges between 
those vertices whose Euclidean distance equals one (called neighbors). 
This is our primal lattice. Its dual lattice is denoted by $(\Z^2)^{\bullet}$. 
For a subgraph $G \subset \Z^2$ (resp.~of $(\Z^2)^{\bullet}$), we define its boundary to be the following set of vertices:
$\partial G = \{ z \in V(G) \, \colon \, \exists \; w \not\in V(G) \text{ such that }\edge{z}{w}\in E(\Z^2)\}$.

Suppose that $G=(V(G),E(G))$ is a finite subgraph of $\mathbb{Z}^2$.  
The \textit{Ising model} on $G$ is a random assignment $\sigma=(\sigma_v)_{v\in V(G)}\in \{\ominus,\oplus\}^{V(G)}$ of spins. The boundary condition $\tau$ is specified by three disjoint subsets $\{\oplus\}$, $\{\ominus\}$ and $\{\mathrm{free}\}$ of $\partial G$, each containing an arbitrary number of arcs.
With boundary condition $\tau$,  and inverse-temperature $\beta>0$, the probability measure of the Ising model is given by
\[\mu_{\beta,G}^{\tau}[\sigma]=\frac{\exp(\beta \sum_{v\sim w}\sigma_v\sigma_w)}{Z_{\beta,G}^{\tau}},\quad Z_{\beta,G}^{\tau}:=\sum_{\sigma}\exp(\beta \sum_{v\sim w}\sigma_v\sigma_w) \]
where the sum in the exponential is taken over the set of edges $\edge{v}{w}\in E(G)$, except for those edges that belong to the $\{\mathrm{free}\}$ arcs in $\tau$, and where the spins at vertices of $\{\oplus\}$ and $\{\ominus\}$ arcs are assumed to be non-random and are equal to $\oplus$ and $\ominus$, respectively.
In this article, we focus on Ising model with critical inverse-temperature
$\beta=\beta_c:=\frac{1}{2}\log(1+\sqrt{2})$. 

When we add the subscript or superscript $\delta$, we mean that subgraphs of the lattices $\Z^2, (\Z^2)^{\bullet}$ have been scaled by $\delta > 0$. We consider the models in the scaling limit $\delta \to 0$. 

\paragraph*{Boundary conditions.}  
For a given dual graph $\mathcal{A}^{\delta, \bullet} \subset (\delta \Z^2)^{\bullet}$,
let $\mathcal{A}^{\delta}\subset\delta\Z^2$ be the graph on the primal lattice corresponding to $\mathcal{A}^{\delta, \bullet}$.  
By a {(discrete) annular polygon} we either refer to a finite doubly connected graph $\mathcal{A}^{\delta, \bullet}$ endowed with given distinct points $x_1^{\delta, \bullet}, \ldots, x_{2N}^{\delta, \bullet}$ on the outer boundary of $\mathcal{A}^{\delta,\bullet}$  in counterclockwise order,
or to the corresponding primal graph  
$(\mathcal{A}^{\delta}; x_1^{\delta}, \ldots, x_{2N}^{\delta})$ 
with given boundary points $x_1^{\delta}, \ldots, x_{2N}^{\delta}$ on the outer boundary of $\mathcal{A}^{\delta}$ in counterclockwise order, see more detail in the  beginning of Section~\ref{sec::Ising_model}. We denote by $\intb^{\delta}$ the inner boundary of $\mathcal{A}^{\delta,\bullet}$. 
Consider the Ising model on a dual annular polygon $(\mathcal{A}^{\delta,\bullet};x_{1}^{\delta,\bullet},\ldots,x_{2N}^{\delta,\bullet})$ with the following boundary conditions: 
	\begin{align} \label{eqn::bounda_condi}
	\oplus\text{ on }\cup_{j=1}^N(x_{2j-1}^{\delta,\bullet}x_{2j}^{\delta,\bullet}),\quad \ominus\text{ on }\cup_{j=1}^{N}(x_{2j}^{\delta,\bullet} x_{2j+1}^{\delta,\bullet}),\quad \text{and}\quad \text{$\free$ on } \intb^{\delta}.
	\end{align}
For each $j\in \{1,2\ldots, 2N\}$, let $\eta_j^\delta$ be the interface starting from $x_j^\delta$. The interface $\eta_j^\delta$ may terminate at some marked points or terminate in the free boundary component $\intb^\delta$.

\paragraph*{Scaling limits.} 
	To specify in which sense the convergence as $\delta \to 0$ should take place, we need a notion of convergence of  polygons. In contrast to the commonly used Carath\'{e}odory convergence of planar sets, we need a slightly stronger notion termed close-Carath\'{e}odory convergence, see the beginning of  Section~\ref{sec::Ising_model} for precise definition.
	Roughly speaking, the usual Carath\'{e}odory convergence allows wild behavior of the boundary approximations, 
	while in order to obtain precompactness of the random interfaces, a slightly stronger convergence which guarantees good approximations around the marked boundary points is required.

	We also need a topology for the interfaces. 
	We denote by $X$ the set of continuous mappings from $[0,1]$ to $\C$ modulo reparameterization, i.e., planar oriented curves. We endow $X$ with metric
	\begin{align} \label{eq::curve_metric_1} 
		\metric(\eta, \tilde{\eta}):=
		\inf_{\psi, \tilde{\psi}} \sup_{t\in[0,1]} |\eta(\psi(t))-\tilde{\eta}(\tilde{\psi}(t))| ,
	\end{align}
	where  the  infimum  is   taken  over  all  increasing  homeomorphisms $\psi, \tilde{\psi} \colon [0,1]\to [0,1]$. 
	Then,  the metric space $(X, \metric)$ is
	complete and separable (see e.g.,~\cite[Section~3.3.1]{KarrilaMultipleSLELocalGlobal}). For $N\geq 1$, we endow $X^{2N}$ with the metric
	\begin{equation}\label{eq::curve_metric_2} 
		\metric\left(\left(\eta_j\right)_{1\leq j\leq 2N}, \left(\tilde{\eta}_j\right)_{1\leq j\leq 2N}\right):=\sup_{1\leq j\leq 2N}\metric\left(\eta_j,\tilde{\eta}_j\right).
	\end{equation}
	Then the metric space $(X^{2N},\metric)$ is also complete and separable.

\paragraph*{Radial Loewner chain.}
To describe scaling limits of interfaces, we recall the notions of radial Loewner chain following the convention in~\cite{HealeyLawlerNSidedRadialSLE}. In this article, we only consider radial Loewner chain which is either a simple curve or a finite union of disjoint simple curves.

Fix $a>0$. Fix $\theta\in [0,\pi)$. 
Suppose $\eta: [0, T]\to \overline{\unitD}$ is a continuous simple curve such that $\eta(0)=z=\exp(2\ii \theta)$ and $\eta(0,T)\subset\unitD \setminus \{0\}$. Let $D_t$ be the connected component of $\unitD\setminus\eta([0,t])$ containing the origin. Let $g_t: D_t\to\unitD$ be the unique conformal map with $g_t(0)=0$ and $g_t'(0)>0$. We say that the curve is parameterized by radial capacity if $g_t'(0)=\exp(2at)$. Then $g_t$ satisfies
\[\partial_t g_t(w)=2ag_t(w)\frac{\exp(2\ii \xi_t)+g_t(w)}{\exp(2\ii \xi_t)-g_t(w)},\quad g_0(w)=w,\quad \forall w\in \unitD\setminus\eta[0,t].\]
 Radial $\SLE_{\kappa}$ is the radial Loewner chain with $\xi_t=B_t$ which is a standard Brownian motion and $\kappa=2/a$. 
 
Next, we introduce notations for $2N$-tuple radial curves. 
For $\boldsymbol{t}=(t_1, \ldots, t_{2N})\in\R_+^{2N}$,
suppose $\{\eta_j([0,t_j]), 1\le j\le 2N\}$ are disjoint. 
Let $D_{\boldsymbol{t}}$ be the connected component of $\unitD\setminus (\cup_j \eta_j([0,t_j]))$ containing the origin. Let $g_{\boldsymbol{t}}:D_{\boldsymbol{t}}\to \unitD$ be the unique conformal transformations with $g_{\boldsymbol{t}}(0)=0$ and $g_{\boldsymbol{t}}'(0)>0$.
We set 
\begin{equation*}
\aleph(\boldsymbol{t})=\log g_{\boldsymbol{t}}'(0).
\end{equation*}
We say that the collection of curves $(\eta_j([0,t_j]), 1\leq j\leq 2N)$ have \textit{$a$-common  parameterization} if for each $t<\tau:=\min_{1\leq j\leq 2N} t_j$, 
\begin{align*}
	\partial_j \aleph(t,t,\ldots,t)=2a,\quad j\in \{1,2,\ldots, 2N\}. 
\end{align*}
In particular, under $a$-common parameterization, we denote $g_t:=g_{(t,t,\ldots,t)}$ and then 
\begin{align*}
	g_t'(0)=\exp(4aNt).
\end{align*}
Moreover, under $a$-common parameterization, $g_t$ satisfies the following radial Loewner equation (see, e.g.,~\cite[Proposition~3.1]{HealeyLawlerNSidedRadialSLE})
\begin{align*}
	\partial_t g_t(w)=2ag_t(w)\sum_{j=1}^{2N}\frac{\exp(2\ii\theta_t^j)+g_t(w)}{\exp(2\ii \theta_t^j)-g_t(w)},\quad \forall w\in \mathbb{D}\setminus\left(\cup_j \eta_j([0,t])\right),\quad \forall t<\tau.
\end{align*}
We say that $\{(\eta_j(t), 0\le t<\tau )\}_{1\leq j\leq 2N}$ is the radial Loewner chain with $a$-common parameterization and with driving functions $\{(\theta_t^j, 0\leq t<\tau)\}_{1\leq j\leq 2N}$.

Now, we are ready to state the conclusion about scaling limit of Ising interfaces in annulus. 
Fix $p>0$ and $\theta^1<\cdots<\theta^{2N}<\theta^1+\pi$ and write $\bs{\theta}=(\theta^1, \ldots, \theta^{2N})$. Denote $x_j=\exp(2\ii \theta^j)$ for $1\leq j\leq 2N$ 
and write $\bs{x}=(x_1, \ldots, x_{2N})$. 
Consider the annular polygon $(\mathbb{A}_{p};x_1,\ldots,x_{2N})$. Suppose that a sequence $(\mathcal{A}^{\delta};x_1^{\delta},\ldots,x_{2N}^{\delta})$ of discrete annular polygons converges to $(\mathbb{A}_{p};x_1,\ldots, x_{2N})$ in the close-Carath\'{e}odory sense (see Section~\ref{sec::Ising_model}). Consider the critical Ising model on $\mathcal{A}^{\delta,\bullet}$ with boundary conditions~\eqref{eqn::bounda_condi}.   For each $j\in \{1,2\ldots, 2N\}$, let $\eta_j^\delta$ be the interface starting from $x_j^\delta$. Using the techniques from~\cite{IzyurovIsingMultiplyConnectedDomains} and~\cite{CHI22}, the collection of interfaces $(\eta_1^\delta,\ldots,\eta_{2N}^\delta)$ converge weakly under the topology induced by~\eqref{eq::curve_metric_2} to a collection of continuous simple curves $(\eta_1,\ldots, \eta_{2N})$, and we denote by $\PP_{\Ising}^{(\mathbb{A}_p; \bs{x})}$ the limiting distribution (see Proposition~\ref{prop::joint_Ising_annulus}). Under $\PP_{\Ising}^{(\mathbb{A}_p; \bs{x})}$, the curves $(\eta_1, \ldots, \eta_{2N})$ may meet each other or hit the inner hole, see Figure~\ref{fig::IsingAnnulus}. As $p\to\infty$, it is less likely that all $2N$ interfaces get close to the inner hole before they meet each other. 
Our first main conclusion is about this rare event: we first derive the asymptotic of the probability of this rare event, and we then derive the limiting distribution as $p\to\infty$ of the curves $(\eta_1, \ldots, \eta_{2N})$ conditional on the rare event. 

\begin{theorem}\label{thm::Ising_annulus_radialSLE}
		Fix $\theta^1<\cdots<\theta^{2N}<\theta^1+\pi$ and write $\bs{\theta}=(\theta^1, \ldots, \theta^{2N})$. Denote $x_j=\exp(2\ii\theta^j)$ for $1\leq j\leq 2N$ and write $\bs{x}=(x_1, \ldots, x_{2N})$. For $p>0$, denote by $\PP_{\Ising}^{(\mathbb{A}_p; \bs{x})}$ the law of the limit $(\eta_1, \ldots, \eta_{2N})$ as described above (see Proposition~\ref{prop::joint_Ising_annulus}). 
\begin{enumerate}
\item[(i)] There exists a constant $C\in (1,\infty)$ such that 
\begin{equation}\label{eqn::IsingAsy}
			C^{-1}\exp\left(-\frac{16N^2-1}{24}p\right)\le \PP_{\Ising}^{(\mathbb{A}_p; \bs{x})}\left[\eta_1, \ldots, \eta_{2N}\text{ all hit } e^{-p}\unitD\right]\le C\exp\left(-\frac{16N^2-1}{24}p\right). 
		\end{equation}		
\item[(ii)] Denote by $\PP_{\mathrm{Ising}\cond p}^{(\mathbb{A}_p;\bs{x})}$ the law of the limit $(\eta_1, \ldots, \eta_{2N})$ conditional on the event $\{\eta_1, \ldots, \eta_{2N}\text{ all hit }e^{-p}\unitD\}$. Denote by $\PP_{\kappa=3\cond *}^{(\bs{\theta})}$ the law of $2N$-sided radial $\SLE_{3}$ in polygon $(\unitD; \exp(2\ii\theta^1), \ldots, \exp(2\ii\theta^{2N}))$ introduced by~\cite{HealeyLawlerNSidedRadialSLE}, see Section~\ref{subsec::pre_2NradialSLE}. We parameterize the process $\boldsymbol{\eta}=(\eta_1, \ldots, \eta_{2N})$ by $\frac{2}{3}$-common parameterization and denote by $(\LF_t, t\ge 0)$ the filtration generated by $(\boldsymbol{\eta}(t), t\ge 0)$. 
		Then, for any $s>0$, when both measures  $\PP_{\mathrm{Ising}\cond p}^{(\mathbb{A}_p;\bs{x})}$ and $\PP_{\kappa=3\cond *}^{(\bs{\theta})}$ are restricted to $\LF_s$, the total variation distance between the two measures goes to zero as $p\to\infty$: 
		\begin{equation}\label{eqn::IsingAnnulusCvg}
			\lim_{p\to \infty}\dist_{TV}\left(\PP_{\Ising\cond p}^{(\mathbb{A}_p; \bs{x})}[\cdot \cond_{\LF_s}], \PP_{\kappa=3\cond *}^{(\bs{\theta})}[\cdot\cond_{\LF_s}]\right)=0.
		\end{equation}
\item[(iii)] Denote by $\nu_{\mathrm{Ising}\cond p}^{(\bs{\theta})}$ the law of the driving function $(\bs{\theta}_t=(\theta_t^1, \ldots, \theta_t^{2N}), t\ge 0)$ under $\PP_{\Ising\cond p}^{(\mathbb{A}_p; \bs{x})}$. Denote by $\nu_{\kappa=3\cond *}^{(\bs{\theta})}$ the law of 2N-radial Bessel process: 
\begin{equation}\label{eqn::RadialBessel}
\ud \theta_t^j=\ud B_t^j+\frac{4}{3}\sum_{k\neq j}\cot(\theta_t^j-\theta_t^k)\ud t, \quad 1\le j\le 2N, 
\end{equation}
where $\{B^j\}_{1\le j\le 2N}$ are independent standard Brownian motions. 
Denote by $(\LG_t, t\ge 0)$ the filtration generated by $(\bs{\theta}_t, t\ge 0)$. 
Then, for any $s>0$, when both measures $\nu_{\mathrm{Ising}\cond p}^{(\bs{\theta})}$ and $\nu_{\kappa=3\cond *}^{(\bs{\theta})}$ are restricted to $\LG_s$, the total variation distance between the two measures goes to zero as $p\to\infty$: 
\begin{equation}\label{eqn::IsingAnnulusDrivingCvg}
\lim_{p\to \infty}\dist_{TV}\left(\nu_{\mathrm{Ising}\cond p}^{(\bs{\theta})}[\cdot \cond_{\LG_s}], \nu_{\kappa=3\cond *}^{(\bs{\theta})}[\cdot\cond_{\LG_s}]\right)=0.
\end{equation}
\end{enumerate}		
	\end{theorem}
The proof of Theorem~\ref{thm::Ising_annulus_radialSLE} relies crucially on an estimate of multiple chordal $\SLE$ in Theorem~\ref{thm::NChordalSLEGreenFunction} and we will explain our strategy to prove Theorem~\ref{thm::Ising_annulus_radialSLE} after the statement of Corollary~\ref{cor::chordaltoradial}. Here we give several remarks for Theorem~\ref{thm::Ising_annulus_radialSLE}: 
\begin{itemize}
\item The number $\frac{16N^2-1}{24}$ in~\eqref{eqn::IsingAsy} is the so-called $2N$-arm exponent of Ising model derived in~\cite{WuAlternatingArmIsing}. We will give a different proof of~\eqref{eqn::IsingAsy} in Section~\ref{sec::cvg_Ising_radial}, where the main ingredient is estimate of multiple chordal $\SLE_{\kappa}$ (Theorem~\ref{thm::NChordalSLEGreenFunction}). 
\item The invariant density for~\eqref{eqn::RadialBessel} is the so-called Dyson's circular ensemble~\cite{DysonJMP1962}. The convergence in~\eqref{eqn::IsingAnnulusDrivingCvg} goes back to~\cite{CardySLEDysonCircularEnsemble, CardySLEDysonCircularEnsembleCorrigendum}. 
\item The conclusions in Theorem~\ref{thm::Ising_annulus_radialSLE} hold regardless of the boundary conditions of the Ising model on the inner boundary. Suppose we consider the Ising model in annular polygon with boundary conditions~\eqref{eqn::bounda_condi} on the outer boundary and another boundary condition in the inner boundary, for instance, any finite combination of $\oplus/\ominus/\free$. Using the same strategy in Section~\ref{sec::Ising_annulus}, we are able to derive the analogue conclusion for Proposition~\ref{prop::joint_Ising_annulus} where the partition function $\partiIsing$ needs to be adjusted to the new boundary condition in the inner boundary. 
Whereas, the analogue conclusions for Theorem~\ref{thm::Ising_annulus_radialSLE} still hold with exactly the same proof.
\end{itemize}

\subsection{Green's function for multiple chordal SLE with $\kappa\in (0,4]$}
\label{subsec::intro_Green}
In this section, we focus on multiple chordal SLE. The motivation to study such process is to describe the scaling limit of interfaces in polygons.
We call $(\Omega; x_1, \ldots, x_{2N})$ a polygon if $\Omega\subset\C$ is a simply connected domain such that $\partial\Omega$ is locally connected and $x_1, \ldots, x_{2N}$ are $2N$ marked points lying on $\partial\Omega$ in counterclockwise order. We denote by $(x_1x_2)$ the boundary arc from $x_1$ to $x_2$ in counterclockwise order. 
We consider $N$ non-crossing simple curves in $\Omega$ such that each curve connects two points among $\{x_1, \ldots, x_{2N}\}$. These curves can have various planar connectivities. We describe such connectivities by link patterns $\alpha=\{\{a_1, b_1\}, \ldots, \{a_N, b_N\}\}$ where $\{a_1, b_1, \ldots, a_N, b_N\}=\{1, 2, \ldots, 2N\}$, and denote by $\LP_N$ the set of all link patterns. Note that $\#\LP_N$ is the Catalan number $C_N=\frac{1}{N+1}\binom{2N}{N}$. 
When $N=1$, we denote by $X(\Omega; x_1, x_2)$ the set of continuous simple unparameterized curves in $\Omega$ connecting $x_1$ and $x_2$ such that they only touch the boundary $\partial\Omega$ in $\{x_1, x_2\}$. In general, for $\alpha=\{\{a_1, b_1\}, \ldots, \{a_N, b_N\}\}\in\LP_N$, we denote by $X_{\alpha}(\Omega; x_1, \ldots, x_{2N})$ the set of families $(\gamma_1, \ldots, \gamma_N)$ of pairwise disjoint curves where $\gamma_j\in X(\Omega; x_{a_j}, x_{b_j})$ for all $j\in\{1,\ldots, N\}$. 

Fix $\kappa\in (0,4]$. For $N\ge 1$ and $\alpha\in\LP_N$, there exists a unique probability measure on $(\gamma_1, \ldots, \gamma_N)\in X_{\alpha}(\Omega; x_1, \ldots, x_{2N})$ such that, for each $j\in\{1, \ldots, N\}$, the conditional law of $\gamma_j$ given $\{\gamma_1, \ldots, \gamma_{N}\}\setminus\{\gamma_j\}$ is the chordal $\SLE_{\kappa}$ connecting $x_{a_j}$ and $x_{b_j}$ in the connected component of the domain $\Omega\setminus\bigcup_{k\neq j}\gamma_k$ having $x_{a_j}$ and $x_{b_j}$ on its boundary. See~\cite[Theorem~1.2]{BeffaraPeltolaWuUniqueness}.  We call this measure \textit{chordal $N$-$\SLE_{\kappa}$} in polygon $(\Omega; x_1, \ldots, x_{2N})$ associated to link pattern $\alpha$. 
Our second main conclusion is the following estimate of chordal $N$-$\SLE_{\kappa}$. 

\begin{theorem}\label{thm::NChordalSLEGreenFunction}
	Fix $\kappa\in (0,4]$. Fix $\theta^1<\cdots<\theta^{2N}<\theta^1+\pi$ and write $\bs{\theta}=(\theta^1, \ldots, \theta^{2N})$. We denote $x_j=\exp(2\ii\theta^j)$ for $1\le j\le 2N$ and write $\bs{x}=(x_1, \ldots, x_{2N})$. 
We denote by $\PP_{\alpha}^{(\unitD; \bs{x})}$ the law of chordal $N$-$\SLE_{\kappa}$ in polygon $(\unitD; \exp(2\ii\theta^1), \ldots, \exp(2\ii\theta^{2N}))$ associated to link pattern $\alpha\in\LP_N$ and suppose $(\gamma_1, \ldots, \gamma_N)\sim\PP_{\alpha}^{(\unitD; \bs{x})}$. We have  
		\begin{equation*}
		\PP_{\alpha}^{(\unitD; \bs{x})}\left[\dist(0,\gamma_j)< r, 1\leq j\leq N\right]=CG_\alpha(\bs{\theta})r^{\armexp}(1+O(r^{u})),\quad \text{as }r\to 0+,
		\end{equation*}
		where $\dist$ is Euclidean distance and 
\begin{itemize}
\item $A_{2N}$ is $2N$-arm exponent: 
			\begin{equation}\label{eqn::armexponent}
		\armexp=\frac{16N^2-(4-\kappa)^2}{8\kappa},
	\end{equation}
\item $G_{\alpha}$ is Green's function for chordal $N$-$\SLE_{\kappa}$ in Definition~\ref{def::NSLEGreen}, 
\item $C\in (0,\infty)$ is a constant depending on $\kappa, N, \alpha$ and $u>0$ is a constant depending on $\kappa, N$. 
\end{itemize}

\end{theorem}
	
\begin{definition}
	\label{def::NSLEGreen}
	Fix $\kappa\in (0,4]$. 
	For $\theta^1<\cdots<\theta^{2N}<\theta^1+\pi$ and $\alpha\in\LP_N$, we define 
	Green's function for chordal $N$-$\SLE_{\kappa}$ as
	\begin{align*}
		G_{\alpha}(\theta^1,\ldots,\theta^{2N}):=\frac{\LG_*(\theta^1,\ldots,\theta^{2N})}{\LG_\alpha(\theta^1,\ldots,\theta^{2N})},
	\end{align*}
	where $\LG_*$ is the partition function for $2N$-sided radial $\SLE_{\kappa}$: 
	\begin{align}\label{eqn::PF_multiplerSLE}
	\LG_*(\theta^1,\ldots,\theta^{2N}):=\prod_{1\leq j<k\leq 2N}|\sin(\theta^k-\theta^j)|^{2/\kappa},
\end{align}
 and $\LG_{\alpha}$ is the pure partition function for chordal $N$-$\SLE_{\kappa}$ defined in~\eqref{eqn::PPF_unitD}. 
\end{definition}

Theorem~\ref{thm::NChordalSLEGreenFunction} is proved for $N=1$ in~\cite{LawlerRezaeiNaturalParameterization, LawlerRezaeiMinkowskiContent} and is proved for $N=2$ in~\cite{ZhanGreen2SLE}; and we will prove it for general $N\ge 3$ and $\alpha\in\LP_N$ in Section~\ref{sec::multipleSLE}. The main ingredients of the proof are properties of chordal $N$-$\SLE_{\kappa}$ from~\cite{PeltolaWuGlobalMultipleSLEs} (see Section~\ref{subsec::pre_chordalNSLE}) and properties of $2N$-sided radial $\SLE_{\kappa}$ from~\cite{HealeyLawlerNSidedRadialSLE} (see Section~\ref{subsec::pre_2NradialSLE}). With these two inputs, we may derive the conclusion using a similar analysis as in~\cite{LawlerRezaeiNaturalParameterization, LawlerRezaeiMinkowskiContent}.  Although the proof strategy is pretty clear, the proof itself still involves a non-trivial calculation, see Section~\ref{subsec::TransitionDensity}. 
We remark that~\cite{LawlerRezaeiNaturalParameterization, LawlerRezaeiMinkowskiContent} and~\cite{ZhanGreen2SLE} prove the conclusion with $N=1$ and $N=2$ for $\kappa\in (0,8)$, whereas we prove the conclusion for $N\ge 3$ and for $\kappa\in (0,4]$. We are not able to extend our conclusion to $\kappa>4$ because the tools on chordal $N$-$\SLE_{\kappa}$ from~\cite{PeltolaWuGlobalMultipleSLEs} and the tools on  $2N$-sided radial $\SLE_{\kappa}$ from~\cite{HealeyLawlerNSidedRadialSLE} are only available for $\kappa\le 4$ so far. 

As a consequence of Theorem~\ref{thm::NChordalSLEGreenFunction}, we are able to describe the law of chordal $N$-$\SLE_{\kappa}$ conditioned to go through the origin. Suppose $(\gamma_1, \ldots, \gamma_N)\sim\PP_{\alpha}^{(\unitD; \bs{x})}$ is chordal $N$-$\SLE_{\kappa}$ in polygon $(\unitD; x_1, \ldots, x_{2N})$ associated to link pattern $\alpha=\{\{a_1, b_1\}, \ldots, \{a_N, b_N\}\}\in \LP_N$. Recall that $\gamma_j$ is a curve in $\unitD$ from $x_{a_j}$ to $x_{b_j}$. We view $(\gamma_1, \ldots, \gamma_N)$ as a $2N$-tuple of continuous simple curves $\bs{\eta}=(\eta_1, \ldots, \eta_{2N})$: for $j\in\{1, \ldots, N\}$, we define $\eta_{a_j}$ to be $\gamma_j$ and $\eta_{b_j}$ to be the time-reversal of $\gamma_j$. In this way, $\eta_j$ is a continuous simple curve in $\unitD$ starting from $x_j$ for $j\in\{1, \ldots, 2N\}$. 

\begin{corollary}\label{cor::chordaltoradial} Fix $\kappa\in (0,4]$ and $a=2/\kappa$ and $\alpha\in\LP_N$. Fix $\theta^1<\cdots<\theta^{2N}<\theta^1+\pi$ and write $\bs{\theta}=(\theta^1, \ldots, \theta^{2N})$. We denote $x_j=\exp(2\ii\theta^j)$ for $1\le j\le 2N$ and write $\bs{x}=(x_1, \ldots, x_{2N})$. 
\begin{itemize}
\item For $p>0$, denote by $\PP_{\alpha\cond p}^{(\unitD; \bs{x})}$ the measure $\PP_{\alpha}^{(\unitD; \bs{x})}$ conditional on the event $\{\eta_1, \ldots, \eta_{2N}\text{ all hit }e^{-p}\unitD\}$. 
\item Denote by $\PP_*^{(\bs{\theta})}$ the law of $2N$-sided radial $\SLE_{\kappa}$ in polygon $(\unitD; \exp(2\ii\theta^1), \ldots, \exp(2\ii\theta^{2N}))$, see Section~\ref{subsec::pre_2NradialSLE}. 
\end{itemize}
We parameterize the process $\bs{\eta}=(\eta_1, \ldots, \eta_{2N})$ by $a$-common parameterization and denote by $(\LF_t, t\ge 0)$ the filtration generated by $(\bs{\eta}(t), t\ge 0)$. Then, for any $s>0$, when both measures $\PP_{\alpha\cond p}^{(\unitD; \bs{x})}$ and $\PP_*^{(\bs{\theta})}$ are restricted to $\LF_{s}$, the total variation distance between the two measures goes to zero as $p\to \infty$:
\begin{equation}\label{eqn::TV_Palpha_P*}
	\lim_{p\to \infty}\emph{dist}_{\text{TV}}\left(\PP_{\alpha\cond p}^{(\unitD; \bs{x})}\left[\cdot |_{\LF_{s}}\right], \PP_*^{(\bs{\theta})}\left[\cdot |_{\LF_{s}}\right]\right)=0.
\end{equation}
Furthermore, we also have the following convergence of the driving functions. Denote by $(\bs{\theta}_t,t\geq 0)$ the driving function  of $(\bs{\eta}(t),t\geq 0)$ under $\PP_{\alpha}^{(\unitD;\bs{x})}$, it satisfies the SDE
	\begin{equation}\label{eqn::SDE_chordalNSLE}
		\ud\theta_t^j=\ud B_t^j+(\partial_j\log\LG_\alpha)(\theta_t^1,\ldots,\theta_t^{2N})\ud t+a\sum_{k\ne j}\cot(\theta_t^j-\theta_t^k)\ud t, \quad 1\leq j\leq 2N, \quad 0\leq t<T,
	\end{equation}
	where $\{B^j\}_{1\leq j\leq 2N}$ are independent standard Brownian motions and $T$ is the lifetime, and we denote by $(\LG_t,t\geq 0)$ the filtration generated by $(\bs{\theta}_t,t\geq 0)$. 
	\begin{itemize}
		\item For $t>0$, denote by $\nu_{\alpha\cond t}^{(\bs{\theta})}$ the law of driving function under $\PP_\alpha^{(\unitD;\bs{x})}$ conditional on $\{T>t\}$.
		\item Denote by $\nu_*^{(\bs{\theta})}$ the law of $2N$-radial Bessel process as in~\cite{HealeyLawlerNSidedRadialSLE}: 
		\begin{equation}\label{eq::2NRadialBES}
	\ud\theta_t^j=\ud B_t^j+2a\sum_{k\ne j}\cot(\theta^j_t-\theta^k_t)\ud t, \quad 1\leq j\leq 2N,
\end{equation}
where $\{B^j\}_{1\leq j\leq 2N}$ are independent standard Brownian motions. 
	\end{itemize}
	Consequently, for any $s>0$, when both measures $\nu_{\alpha\cond t}^{(\bs{\theta})}$ and $\nu_*^{(\bs{\theta})}$ are restricted to $\LG_s$, the total variation distance between the two measures goes to zero as $t\to\infty$:
	\begin{equation}\label{eqn::dist_TV_nu}
		\lim_{t\to\infty}\dist_{TV}\left(\nu_{\alpha\cond t}^{(\bs{\theta})}[\cdot\cond_{\LG_s}], \nu_*^{(\bs{\theta})}[\cdot\cond_{\LG_s}]\right)=0.
\end{equation}
\end{corollary}

We see from Corollary~\ref{cor::chordaltoradial} that $2N$-sided radial $\SLE_{\kappa}$ can be viewed as chordal $N$-$\SLE_{\kappa}$ conditioned to the event that all curves go through the origin, and that the driving function of chordal $N$-$\SLE_\kappa$ as a radial Loewner chain with $a$-common parameterization conditioned on the same event converges to $2N$-radial Bessel process, the driving function of $2N$-sided radial $\SLE_\kappa$. 

Now, let us explain how we prove Theorem~\ref{thm::Ising_annulus_radialSLE}. 
Fix $\theta^1<\cdots<\theta^{2N}<\theta^1+\pi$ and write $\bs{\theta}=(\theta^1, \ldots, \theta^{2N})$. 
We denote $x_j=\exp(2\ii\theta^j)$ for $1\le j\le 2N$ and write $\bs{x}=(x_1, \ldots, x_{2N})$. 
Recall that $\PP_{\Ising}^{(\mathbb{A}_p; \bs{x})}$ denotes the law of the limit $(\eta_1, \ldots, \eta_{2N})$ of Ising interfaces in annulus, as in Proposition~\ref{prop::joint_Ising_annulus}. We will compare $\PP_{\mathrm{Ising}}^{(\mathbb{A}_p;\bs{x})}$ with the scaling limit of Ising interfaces in polygons $(\unitD; x_1, \ldots, x_{2N})$, whose law is denoted by $\PP_{\Ising}^{(\unitD;\bs{x})}$. 
We denote by $\PP_{\kappa=3\cond \alpha}^{(\unitD; \bs{x})}$ the chordal $N$-$\SLE_3$ in polygon $(\unitD; x_1, \ldots, x_{2N})$ associated to link pattern $\alpha\in\LP_N$. Then the limit $\PP_{\Ising}^{(\unitD;\bs{x})}$ is a linear combination of $\{\PP_{\kappa=3\cond\alpha}^{(\unitD; \bs{x})}\}_{\alpha\in\LP_N}$, as proved in~ \cite{IzyurovIsingMultiplyConnectedDomains, BeffaraPeltolaWuUniqueness, PeltolaWuCrossingProbaIsing, KarrilaNewProba}, see Proposition~\ref{prop::Ising_polygon}. 
The analogue of Theorem~\ref{thm::Ising_annulus_radialSLE} for $\PP_{\Ising}^{(\unitD;\bs{x})}$ holds due to Theorem~\ref{thm::NChordalSLEGreenFunction} and Corollary~\ref{cor::chordaltoradial}. Then we show that Theorem~\ref{thm::Ising_annulus_radialSLE} also hold for $\PP_{\Ising}^{(\mathbb{A}_p;\bs{x})}$ by comparing it with $\PP_{\Ising}^{(\unitD;\bs{x})}$. This is accomplished by constructing a good coupling between the two measures, see Section~\ref{sec::cvg_Ising_radial}.

\section{Multiple chordal SLE and multiple radial SLE}
\label{sec::multipleSLE}
In this section, we consider the relation between multiple chordal SLE and multiple radial SLE. 
We fix the following parameters throughout this section:  
\begin{align}\label{eqn::constatnsabc}
	\kappa\in (0,4],\qquad	a=\frac{2}{\kappa},\qquad b=\frac{6-\kappa}{2\kappa},\qquad \tilde{b}=\frac{(\kappa-2)(6-\kappa)}{8\kappa},\qquad c=\frac{(6-\kappa)(3\kappa-8)}{2\kappa}.
\end{align}
Furthermore, we fix $N\ge 1$ and $\theta^1<\cdots<\theta^{2N}<\theta^1+\pi$ and write $\bs{\theta}=(\theta^1, \ldots, \theta^{2N})$. 
We denote $x_j=\exp{(2\ii\theta^j)}$ for $1\le j\le 2N$ and denote $\bs{x}=(x_1, \ldots, x_{2N})$. 
Recall that $\PP_{\alpha}^{(\unitD; \bs{x})}$ denotes the law of chordal $N$-$\SLE_{\kappa}$ in polygon $(\unitD; \exp(2\ii\theta^1), \ldots, \exp(2\ii\theta^{2N}))$ associated to link pattern $\alpha\in\LP_N$. 
In this section, we also denote 
\[\PP_{\alpha}^{(\bs{\theta})}=\PP_{\alpha}^{(\unitD; \bs{x})}.\]
We view it as a measure on $2N$-tuples of curves $\bs{\eta}=(\eta_1,\ldots, \eta_{2N})$ as in Section~\ref{subsec::intro_Green}. Recall that we have introduced $a$-common parameterization for $2N$-tuples of curves in Section~\ref{subsec::Ising_intro}. Under $a$-common parameterization, we denote by $T$ the lifetime of $\bs{\eta}=(\eta_1,\ldots,\eta_{2N})$, i.e., the first time that $\bs{\eta}$ hits the origin or that two of the curves in $\bs{\eta}$ meet. Under $\PP_{\alpha}^{(\bs{\theta})}$, the lifetime $T$ is finite almost surely and we have the following estimate.

\begin{proposition}\label{prop::multiplechordalestimates}  
	For $t>0$, 
	\begin{equation}\label{eqn::multiplechordalestimates}
		\PP_{\alpha}^{(\bs{\theta})}[T>t]=\LJ_{\alpha} G_\alpha(\bs{\theta})e^{-\armexp\cdot 4aNt}(1+O(e^{-\cvgexp\cdot 4aNt})), 
	\end{equation}
	where $\LJ_\alpha$ is a normalization constant defined in~\eqref{eq::LJ_alpha}, 
	$G_{\alpha}$ is Green's function for chordal $N$-$\SLE_{\kappa}$ in Definition~\ref{def::NSLEGreen}, and $\armexp$ is the $2N$-arm exponent~\eqref{eqn::armexponent} and $\cvgexp>0$ depends only on $\kappa$ and $N$. 
\end{proposition}

We will give preliminaries on chordal $N$-$\SLE_{\kappa}$ in Section~\ref{subsec::pre_chordalNSLE} and give preliminaries on $2N$-sided radial $\SLE_{\kappa}$ in Section~\ref{subsec::pre_2NradialSLE}. We then derive transition density and quasi-invariant density of chordal $N$-$\SLE_{\kappa}$ in Section~\ref{subsec::TransitionDensity}. With these at hand, we prove Proposition~\ref{prop::multiplechordalestimates} and Theorem~\ref{thm::NChordalSLEGreenFunction} and Corollary~\ref{cor::chordaltoradial} in Section~\ref{subsec::proofs}.

\subsection{Preliminaries: chordal $N$-$\SLE_{\kappa}$}
\label{subsec::pre_chordalNSLE}
In this section, we will describe the Loewner chain of chordal $N$-$\SLE_{\kappa}$. This involves the notion of \textit{pure paritition functions} of multiple $\SLE_{\kappa}$. These are the recursive collection $\{\PartF_{\alpha} \colon \alpha \in \bigsqcup_{N\geq 0} \LP_N\}$ 
of functions \[\PartF_{\alpha} \colon \{(x_1,\ldots,x_{2N})\in\R^{2N}: x_1<\cdots<x_{2N}\}\to\R\]
uniquely determined by the following 4 properties:
\begin{itemize}
	\item BPZ equations: for all $ j \in \{1,\ldots,2N\}$, 
	\begin{align*}
		\left[ 
		\frac{1}{2} \partial_j^2
		+ \sum_{k\neq j} \left( \frac{a}{x_{k}-x_{j}} \partial_k
		- \frac{ab}{(x_{k}-x_{j})^{2}} \right) \right]
		\PartF_{\alpha}(x_1,\ldots,x_{2N}) =  0.
	\end{align*}
	\item M\"{o}bius covariance: for all M\"obius maps $\varphi$ of the upper half-plane $\HH$ such that $\varphi(x_{1}) < \cdots < \varphi(x_{2N})$, we have
	\begin{align}\label{eqn::COV}
		\PartF_{\alpha}(x_{1},\ldots,x_{2N}) = 
		\prod_{j=1}^{2N} \varphi'(x_{j})^{b} 
		\times \PartF_{\alpha}(\varphi(x_{1}),\ldots,\varphi(x_{2N})).
	\end{align}
	\item Asymptotics: with $\PartF_{\emptyset} \equiv 1$ for the empty link pattern $\emptyset \in \LP_0$, the collection $\{\PartF_{\alpha} \colon \alpha\in\LP_N\}$ satisfies the following recursive asymptotics property. Fix $N \ge 1$ and $j \in \{1,2, \ldots, 2N-1 \}$. 
	Then, we have
	\begin{align*}
		\lim_{x_j,x_{j+1}\to\xi} \frac{\PartF_{\alpha}(x_1,\ldots, x_{2N})}{ (x_{j+1}-x_j)^{-2b} }
		= 
		\begin{cases}
			\PartF_{\alpha/\{j,j+1\}}(x_1, \ldots, x_{j-1}, x_{j+2}, \ldots, x_{2N}), 
			& \quad \text{if }\{j, j+1\}\in\alpha , \\
			0 ,
			& \quad \text{if }\{j, j+1\} \not\in \alpha ,
		\end{cases}
	\end{align*}
	where $\xi \in (x_{j-1}, x_{j+2})$ (with the convention that $x_0 = -\infty$ and  $x_{2N+1} = +\infty$), and $\alpha/\{k,l\}$ denotes the link pattern in $\LP_{N-1}$ obtained by removing $\{k,l\}$ from $\alpha$ and then relabeling the remaining indices so that they are the first $2(N-1)$ positive integers. 
	\item The functions are positive and satisfy the following power-law bound:
	\begin{align*}
		0<\PartF_{\alpha}(x_1, \ldots, x_{2N})\le\prod_{\{k,l\}\in\alpha}|x_k-x_l|^{-2b}, \quad \text{for all }x_1<\cdots<x_{2N}. 
	\end{align*}
\end{itemize}

The uniqueness when $\kappa\in (0,8)$ of such collection of functions were proved in~\cite{FloresKlebanPDE2}. The existence when $\kappa\in (0,4]$ were proved in~\cite{PeltolaWuGlobalMultipleSLEs}. For other results related to the existence of such functions, see~\cite{FloresKlebanPDE3, KytolaPeltolaPurePartitionFunctions, WuHyperSLE}. 

We extend the definition of $\PartF_{\alpha}$ to more general
polygons $(\Omega; x_1, \ldots, x_{2N})$ whose marked boundary points $x_1, \ldots, x_{2N}$ lie on sufficiently regular boundary segments (e.g. $C^{1+\eps}$ for some $\eps>0$) as 
\begin{align}\label{eqn::PartF_def_polygon}
	\PartF_{\alpha}(\Omega; x_1, \ldots, x_{2N}):=\prod_{j=1}^{2N}|\varphi'(x_j)|^b\times \PartF_{\alpha}(\varphi(x_1), \ldots, \varphi(x_{2N})),
\end{align}
where $\varphi$ is any conformal map from $\Omega$ onto $\HH$ with $\varphi(x_1)<\cdots<\varphi(x_{2N})$. 
In this section, we are interested in the polygon $(\unitD; \exp(2\ii\theta^1), \ldots, \exp(2\ii\theta^{2N}))$ with $\theta^1<\cdots<\theta^{2N}<\theta^1+\pi$, and we denote 
\begin{align}\label{eqn::PPF_unitD}
	\LG_{\alpha}\left(\theta^1, \ldots, \theta^{2N}\right)=\PartF_{\alpha}\left(\unitD; \exp\left(2\ii\theta^1\right), \ldots, \exp\left(2\ii\theta^{2N}\right)\right).
\end{align}
By direct calculation, $\LG_{\alpha}$ satisfies the following PDE: for all $j\in\{1,\ldots,2N\}$, 
\begin{equation}\label{eqn::PDEcircle}
	\left[\frac{1}{2}\partial_j^2+\sum_{k\neq j}\left(a\cot(\theta^k-\theta^j)\partial_k-ab\csc^2(\theta^k-\theta^j)\right)-2a\tilde{b}\right]\LG_\alpha(\theta^1, \ldots, \theta^{2N})=0,
\end{equation}
and the following power-law bound:
\begin{equation}\label{eq::G_PLB}
	0< \LG_\alpha(\theta^1,\ldots,\theta^{2N})\leq 2^{-2Nb}\prod_{\{k,l\}\in\alpha}|\sin(\theta^k-\theta^l)|^{-2b}, \quad \text{for all } \theta^1<\cdots<\theta^{2N}<\theta^1+\pi.
\end{equation}

Next, we will describe the Loewner evolution of chordal $N$-$\SLE_{\kappa}$. To this end, we first introduce the usual parameterization. 
We fix $N\geq 1$ and let  $\bs{\theta}=(\theta^1,\ldots,\theta^{2N})$ with $\theta^1<\theta^2<\cdots<\theta^{2N}<\theta^1+\pi$. 
Let $\bs{\eta}=(\eta_1,\ldots, \eta_{2N})$ be a $2N$-tuple of continuous simple curves $\eta_j: [0,t_j]\to \overline{\unitD}$ such that $\eta_j(0)=\exp(2\ii \theta^j)$ and $\eta_j(0,t_j)\subset\unitD \setminus\{0\}$. For $j\in \{1,\ldots,2N\}$, let $D_{t}^j$ be the connected components of $\unitD\setminus \eta_j([0,t])$ containing the origin. Let $g_{t}^j: D_{t}^j\to \unitD$ be the unique conformal transformations with $g_{t}^j(0)=0$ and $(g_{t}^j)'(0)>0$. 
We say that $\eta_j=(\eta_j(t), 0\leq t< t_j)$ has \textit{$a$-usual  parameterization} if we have 
\begin{align*}
	(g_t^j)'(0)=\exp(2at), \quad \forall t<t_j.
\end{align*}
Under this parameterization, $g_t^j$ satisfies the following radial Loewner equation (see, e.g.,~\cite[Proposition~3.1]{HealeyLawlerNSidedRadialSLE})
\begin{align*}
	\partial_t {g}_t^j(w)=2a{g}^j_t(w) \frac{\exp(2\ii\xi_t^j)+{g}^j_t(w)}{\exp(2\ii\xi_t^j)-{g}_t^j(w)},\quad \forall w\in \unitD\setminus \eta_j([0,t]),\quad \forall t<t_j.
\end{align*} 
We say that $(\eta_j(t), 0\leq t<t_j)$ is the radial Loewner chain with $a$-usual parameterization and with driving function $(\xi_t^j, 0\leq t<t_j)$. 

Let $(\gamma_1, \ldots, \gamma_N)\sim\PP_{\alpha}^{(\bs{\theta})}$ be the chordal $N$-$\SLE_{\kappa}$ in polygon $(\unitD; \exp(2\ii\theta^1), \ldots, \exp(2\ii\theta^{2N}))$ associated to $\alpha\in\LP_N$. We view it as a $2N$-tuple of continuous simple curves $\bs{\eta}=(\eta_1, \ldots, \eta_{2N})$ as described in Section~\ref{subsec::intro_Green}. 
From~\cite[Proposition~4.10]{PeltolaWuGlobalMultipleSLEs} and a standard argument as~\cite[Theorem~3]{SchrammWilsonSLECoordinatechanges}, for $j\in\{1, \ldots, 2N\}$, if we parameterize $\eta_j$ with $a$-usual parameterization, then its driving function $(\xi^j_t,t\geq 0)$ satisfies 
\begin{equation}\label{eq::marginal}
	\begin{cases}
		\ud\xi_t^j=\ud B_t^j+(\partial_j\log\LG_\alpha)(V^1_t,\ldots,V^{j-1}_t,\xi_t^j,V^{j+1}_t,\ldots,V^{2N}_t)\ud t, \quad \xi_0^j=\theta^j;\\
		\ud V_t^k=a\cot(V_t^k-\xi_t^j)\ud t, \quad V_0^k=\theta^k, \quad k\in\{1,\ldots,j-1,j+1,\ldots,2N\};
	\end{cases}
\end{equation}
where $(B^j_t,t\geq 0)$ is a standard Brownian motion.

\subsection{Preliminaries: $2N$-sided radial $\SLE$}
\label{subsec::pre_2NradialSLE}

Fix $\theta^1<\cdots<\theta^{2N}<\theta^1+\pi$ and write $\bs{\theta}=(\theta^1, \ldots, \theta^{2N})$. Recall from~\eqref{eqn::PF_multiplerSLE} that  
\begin{align*}
	\LG_*(\theta^1,\ldots,\theta^{2N}):=\prod_{1\leq j<k\leq 2N}|\sin(\theta^k-\theta^j)|^{a}.
\end{align*}
We consider $2N$ radial curves $\bs{\eta}=(\eta_1,\ldots,\eta_{2N})$ in the polygon $(\unitD; \exp(2\ii\theta^1), \ldots, \exp(2\ii\theta^{2N}))$ with $a$-common parameterization, and define the following normalized conformal maps:
\begin{itemize}
	\item $g_{t}^j: \unitD\setminus\eta_j([0,t])\to\unitD$ is conformal with $g_{t}^j(0)=0$, $(g_{t}^j)'(0)>0$, $1\leq j\leq 2N$.
	\item $g_{t}: \unitD\setminus \bigcup_{j=1}^{2N}\eta_j([0,t])\to \unitD$ is conformal with $g_{t}(0)=0$, $g_{t}^\prime(0)=\exp(4aNt)>0$. 
	\item $g_{t,j}: \unitD\setminus\bigcup_{k\ne j}g_{t}^{j}(\eta_k([0,t]))\to\unitD$ is conformal with $g_{t,j}(0)=0$, $g_{t,j}^\prime(0)>0$, $1\leq j\leq 2N$. 
\end{itemize}
Then $g_t=g_{t,j}\circ g_{t}^j$ for $1\leq j\leq 2N$. 
Let $h_{t}^j$ be the covering conformal map of $g_{t}^j$, i.e., 
\[g_{t}^j(e^{2\ii\zeta})=\exp\left(2\ii h_{t}^j(\zeta)\right) \text{ with } h_0^j(\zeta)=\zeta \text{ for }\zeta\in\HH. \]
Define $h_{t}$ and $h_{t,j}$ similarly for $g_{t}$ and $g_{t,j}$. 
Denote by $(\xi^j_{t}, t\geq 0)$ the driving function of $\eta_j$ as a radial Loewner curve, and  $(\bs{\theta}_t=(\theta_t^1,\ldots,\theta_t^{2N}),t\geq 0)$ the driving functions of $(\eta_1,\ldots,\eta_{2N})$ as a radial Loewner chain with $a$-common paramterization as described in Section~\ref{subsec::Ising_intro}.

Let $\PP$ denote the probability measure under which $(\eta_1,\ldots,\eta_{2N})$ are $2N$ independent radial $\SLE_\kappa$ in $\unitD$ started from $(\exp(2\ii\theta^1), \ldots, \exp(2\ii\theta^{2N}))$ respectively. 
From~\cite[Proposition~3.11]{HealeyLawlerNSidedRadialSLE}, 
\begin{equation}\label{eqn::mart*}
	M_t^*=g_t'(0)^{\frac{(4N^2-1)a}{4}-(2N-1)\tilde{b}}
	\prod_{j=1}^{2N}h_{t,j}'(\xi_t^j)^b g_{t,j}'(0)^{\tilde{b}}\times \LG_*(\theta^1_t,\ldots,\theta^{2N}_t)\exp\left(\frac{c}{2}\sum_{j=1}^{2N}\mu_t^j\right)
\end{equation}
is a local martingale with respect to $\PP$, 
where $\mu_t^j$ is defined by 
	\begin{equation*}
		\mu_{t}^j=-\frac{a}{6}\int_0^{t}Sh_{s,j}(\xi_{s}^j)\ud s+\frac{a}{3}\int_0^{t}(1-h_{s,j}'(\xi_{s}^j)^2)\ud s,
\end{equation*}
and $Sh:=\left(\frac{h^{\prime\prime}}{h'}\right)'-\frac{1}{2}\left(\frac{h^{\prime\prime}}{h'}\right)^2=\frac{h^{\prime\prime\prime}}{h^\prime}-\frac{3}{2}\left(\frac{h^{\prime\prime}}{h'}\right)^2$ denotes the Schwarzian derivative of $h$. 
Denote by $\PP_*^{(\bs{\theta})}$ the probability measure obtained by tilting $\PP$ by $M_t^*$, then
\begin{equation*}
	\ud\theta_t^j=\ud B_t^j+2a\sum_{k\ne j}\cot(\theta^j_t-\theta^k_t)\ud t, \quad 1\leq j\leq 2N,
\end{equation*}
where $\{B^j\}_{1\leq j\leq 2N}$ are independent standard Brownian motions under $\PP_*^{(\bs{\theta})}$. Note that this is the same SDE as in~\eqref{eq::2NRadialBES}. 

The radial Loewner chain $(\bs{\eta}(t)=(\eta_1(t),\ldots,\eta_{2N}(t)), t\geq 0)$ with $a$-common parameterization and with driving functions $(\bs{\theta}_t=(\theta^1_t,\ldots,\theta^{2N}_t),t\geq 0)$ satisfying SDE~\eqref{eq::2NRadialBES} is called the \textit{$2N$-sided radial $\SLE_\kappa$} in polygon $(\unitD; \exp(2\ii\theta^1), \ldots, \exp(2\ii\theta^{2N}))$, introduced in~\cite{HealeyLawlerNSidedRadialSLE}. 
Recall that the lifetime $T$ is the first time that $\bs{\eta}$ hits the origin or that two of the curves in $\bs{\eta}$ meet. From~\cite[Proposition~5.4]{HealeyLawlerNSidedRadialSLE}, we have
$\PP_*^{(\bs{\theta})}[T=\infty]=1$.
We then recall some conclusions from~\cite{HealeyLawlerNSidedRadialSLE} about the transition density of $(\bs{\theta}_t, t\geq 0)$. 

For $N\geq 1$, denote by $\mathcal{X}'_{2N}$ the torus $[0,\pi)^{2N}$ with periodic boundary conditions and  $\mathcal{X}_{2N}$ the set of $\bs{\theta}=(\theta^1,\ldots,\theta^{2N})\in \mathcal{X}_{2N}'$ such that we can find representatives with $\theta^1<\cdots<\theta^{2N}<\theta^1+\pi$. 
For each $v>0$ and $\bs{\theta}\in\mathcal{X}_{2N}$, define
\begin{equation}\label{eqn::def_F_f}
	F_v(\bs{\theta})=\prod_{1\leq j<k\leq 2N}|\sin(\theta^k-\theta^j)|^v,\quad 
	f_v(\bs{\theta})=\mathcal{I}_v^{-1} F_v(\bs{\theta}), \quad\text{with}\quad \mathcal{I}_v=\int_{\mathcal{X}_{2N}}F_v(\bs{\theta})\ud \bs{\theta}.
\end{equation}
Note that $F_v(\bs{\theta})\in (0,1]$ for $\bs{\theta}\in\mathcal{X}_{2N}$, and $\LG_*(\bs{\theta})$ in~\eqref{eqn::PF_multiplerSLE} is exactly $F_a(\bs{\theta})$. 

\begin{lemma}{\cite[Proposition~5.5]{HealeyLawlerNSidedRadialSLE}}
	\label{lem::density*}
	Fix $\kappa\in(0,4]$ and $a=2/\kappa$. 
	Denote by $p_*(t;\bs{\theta},\cdot)$ the transition density for $(\bs{\theta}_t, t\ge 0)$ under $\PP_*^{(\bs{\theta})}$. Then, for any  $\bs{\theta},\bs{\theta}^\prime\in\mathcal{X}_{2N}$, 
	\begin{equation}\label{eqn::density*_asy}
		p_*(t;\bs{\theta},\bs{\theta}^\prime)=p_*(\bs{\theta}')(1+O(e^{-\cvgexp\cdot 4aNt})), 
	\end{equation}
	where $\cvgexp>0$ is a constant depending on $\kappa, N$ and $p_*(\bs{\theta}')$ is the invariant density for $(\bs{\theta}_t,t\geq 0)$: 
	\begin{equation}\label{eqn::inv_density}
		p_*(\bs{\theta}')=f_{4a}(\bs{\theta}'), \quad \bs{\theta}'\in\mathcal{X}_{2N}.
	\end{equation}
\end{lemma}
The invariant density~\eqref{eqn::inv_density} is called Dyson's circular ensemble in~\cite{CardySLEDysonCircularEnsemble, CardySLEDysonCircularEnsembleCorrigendum}.

\subsection{Transition density for chordal $N$-$\SLE_{\kappa}$}
\label{subsec::TransitionDensity}

For $\bs{\theta}=(\theta^1, \ldots, \theta^{2N})\in\mathcal{X}_{2N}$, recall that $\PP_\alpha^{(\bs{\theta})}$ denotes the law of chordal $N$-$\SLE_{\kappa}$ in the polygon $(\unitD; \exp(2\ii\theta^1), \ldots, \exp(2\ii\theta^{2N}))$ associated to link pattern $\alpha\in\LP_N$. Denote by 
$(\bs{\theta}_t=(\theta^1_t,\ldots,\theta^{2N}_t),0\leq t<T)$ the driving function of chordal $N$-$\SLE_{\kappa}$ as a radial Loewner chain with $a$-common parameterization. 
In this section, we study the growth of $(\bs{\theta}_t,0\leq t<T)$ in $\mathcal{X}_{2N}$ under $\PP_\alpha^{(\bs{\theta})}$. Denote by $p_{\alpha}(t; \cdot, \cdot)$ the transition density. By saying that $p_\alpha(t;\cdot,\cdot)$ is the transition
density, we mean that for any $t>0$ and $\bs{\theta}\in\mathcal{X}_{2N}$, 
if $(\bs{\theta}_t, 0\leq t<T)$ starts from $\bs{\theta}$, then for any bounded measurable function $h$ on $\mathcal{X}_{2N}$, 
\begin{equation*}
	\mathbb{E}_\alpha^{(\bs{\theta})}\left[\mathbb{1}_{\{T>t\}}h(\bs{\theta}_t)\right]=\int_{\mathcal{X}_{2N}}p_\alpha(t;\bs{\theta},\bs{\theta}')h(\bs{\theta}')\ud\bs{\theta}'.
\end{equation*}
In particular,
\begin{equation}\label{eq::Tetsimate}
	\PP_\alpha^{(\bs{\theta})}[T>t]=\int_{\mathcal{X}_{2N}}p_\alpha(t;\bs{\theta},\bs{\theta}')\ud\bs{\theta}'.
\end{equation}

The goal of this section is to derive the transition densities  $p_{\alpha}(t; \bs{\theta}, \bs{\theta}')$. To this end, we first derive the Radon-Nikodym derivative between $\PP_{\alpha}^{(\bs{\theta})}$ and the law of $2N$ independent radial $\SLE_\kappa$.

\begin{lemma}\label{lem::RN_Palpha_P}
	Fix $\theta^1<\cdots<\theta^{2N}<\theta^1+\pi$ and write $\bs{\theta}=(\theta^1, \ldots, \theta^{2N})$. 
	\begin{itemize}
		\item Recall that $\PP_{\alpha}^{(\bs{\theta})}$ denotes the probability measure of chordal $N$-$\SLE_{\kappa}$ in polygon $(\unitD; \exp(2\ii\theta^1), \ldots, \exp(2\ii\theta^{2N}))$ associated to $\alpha\in\LP_N$ and we view it as a measure on $2N$-tuples of curves $\bs{\eta}=(\eta_1, \ldots, \eta_{2N})$. 
		\item Suppose $\PP$ is the probability measure on $\bs{\eta}=(\eta_1, \ldots, \eta_{2N})$ where $(\eta_1, \ldots, \eta_{2N})$ are independent radial $\SLE_{\kappa}$ in $\unitD$ started from $(\exp(2\ii\theta^1), \ldots, \exp(2\ii\theta^{2N}))$ respectively. 
	\end{itemize}	
	Under the $a$-common parameterization and we use the same notations as in Section~\ref{subsec::pre_2NradialSLE}, 
	the Radon-Nikodym derivative between $\PP_\alpha^{(\bs{\theta})}$ and $\PP$ when both measures are restricted to $\LF_t$ and $\{T>t\}$ is 
	\begin{equation*}
		\frac{\ud \PP_\alpha^{(\bs{\theta})}[\cdot|_{\LF_t\cap\{T>t\}}]}{\ud \PP[\cdot|_{\LF_t\cap\{T>t\}}]}
		=\frac{M^\alpha_t}{M^\alpha_0}, \quad t>0,
	\end{equation*}
	where
	\begin{equation}\label{eqn::mart_alpha}
		M_t^{\alpha}=g_t'(0)^{-2N\tilde{b}}\prod_{j=1}^{2N}h_{t,j}'(\xi_t^j)^b g_{t,j}'(0)^{\tilde{b}}\times\LG_\alpha(\theta^1_t,\ldots,\theta^{2N}_t)\exp\left(\frac{c}{2}\sum_{j=1}^{2N}\mu_t^j\right),  \quad 0\leq t<T. 
	\end{equation}
\end{lemma}

To prove Lemma~\ref{lem::RN_Palpha_P}, we need to introduce a local martingale with $2N$ time parameters.
For $\bs{t}=(t_1,\ldots,t_{2N})\in\mathbb{R}_+^{2N}$, suppose $\{\eta_j([0,t_j])\}_{1\leq j\leq 2N}$ are disjoint. 
We consider the following normalized conformal maps:
\begin{itemize}
	\item $g_{t_j}^j: \unitD\setminus\eta_j([0,t_j])\to\unitD$ is conformal with $g_{t_j}^j(0)=0$, $(g_{t_j}^j)'(0)=\exp(2at_j)>0$, $1\leq j\leq 2N$.
	\item $g_{\bs{t}}: \unitD\setminus \bigcup_{j=1}^{2N}\eta_j([0,t_j])\to \unitD$ is conformal with $g_{\bs{t}}(0)=0$, $g_{\bs{t}}^\prime(0)=\exp(\aleph(\bs{t}))>0$.  
	\item $g_{\bs{t},j}: \unitD\setminus\bigcup_{k\ne j}g_{t_j}^{j}(\eta_k([0,t_k]))\to\unitD$ is conformal with $g_{\bs{t},j}(0)=0$, $g_{\bs{t},j}^\prime(0)>0$, $1\leq j\leq 2N$. 
\end{itemize}
Then $g_{\bs{t}}=g_{\bs{t},j}\circ g_{t_j}^j$ for $1\leq j\leq 2N$. 
Let $h_{t_j}^j$, $h_{\bs{t}}$ and $h_{\bs{t},j}$ be the covering conformal maps of $g_{t_j}^j$, $g_{\bs{t}}$ and $g_{\bs{t},j}$ respectively. 
Denote by $(\xi^j_{t_j}, t_j\geq 0)$ the driving function of $\eta_j$ as a radial Loewner chain with $a$-usual parameterization. Then we have radial Loewner equations
\begin{equation*}
	\partial_{t_j} g_{t_j}^j(z)=2ag_{t_j}^j(z)\frac{\exp(2\ii\xi_{t_j}^j)+g_{t_j}^j(z)}{\exp(2\ii\xi_{t_j}^j)-g_{t_j}^j(z)}, \quad \text{and}\quad 
	\partial_{t_j} h_{t_j}^j(\zeta)=a\cot(h_{t_j}^j(\zeta)-\xi_{t_j}^j).
\end{equation*}
Let $\bs{\theta}_{\bs{t}}=(\theta^1_{\bs{t}},\ldots,\theta^{2N}_{\bs{t}})$ with  $\theta^j_{\bs{t}}=h_{\bs{t},j}(\xi_{t_j}^j)$, $1\leq j\leq 2N$.

\begin{lemma}\label{lem::Plocalmart}
	Suppose $\PP$ is the probability measure on $\bs{\eta}=(\eta_1, \ldots, \eta_{2N})$ where $(\eta_1, \ldots, \eta_{2N})$ are independent radial $\SLE_{\kappa}$ in $\unitD$ started from $(\exp(2\ii\theta^1), \ldots, \exp(2\ii\theta^{2N}))$ respectively. Then the process
	\begin{equation}\label{Malpha_boldt}
		M^{\alpha}_{\bs{t}}=g_{\bs{t}}^\prime(0)^{-2N\tilde{b}}\prod_{j=1}^{2N} h_{\bs{t},j}^\prime(\xi_{t_j}^j)^b g_{\bs{t},j}^\prime(0)^{\tilde{b}}\times \LG_\alpha(\theta^1_{\bs{t}},\ldots,\theta^{2N}_{\bs{t}})\exp\left(\frac{c}{2}\sum_{j=1}^{2N} \mu_{\bs{t}}^j\right)
	\end{equation}
	is a $2N$-time-parameter local martingale with respect to $\PP$, where $\mu_{\bs{t}}^j$ is given by 
		\begin{equation*}
			\mu_{\bs{t}}^j=-\frac{a}{6}\int_0^{t_j}Sh_{\bs{t},j}(\xi_{t_j}^j)\ud t_j+\frac{a}{3}\int_0^{t_j}(1-h_{\bs{t},j}'(\xi_{t_j}^j)^2)\ud t_j.
	\end{equation*}
\end{lemma}

\begin{proof}
	Since $\theta_{\bs{t}}^j=h_{\bs{t},j}(\xi_{t_j}^j)$ for $1\leq j\leq 2N$, by It\^{o}'s formula and a standard calculation, we have  
	\begin{equation*}
		\ud \theta_{\bs{t}}^j=h_{\bs{t},j}^\prime(\xi_{t_j}^j)\ud\xi_{t_j}^j-bh_{\bs{t},j}^{\prime\prime}(\xi_{t_j}^j)\ud t_j+a\sum_{k\ne j}\cot(\theta^j_{\bs{t}}-\theta^k_{\bs{t}})h_{\bs{t},k}'(\xi_{t_k}^k)^2\ud t_k, \quad 1\leq j\leq 2N.
	\end{equation*}
	Note that $M_{\bs{t}}^\alpha$ defined in~\eqref{Malpha_boldt} has a decomposition $M_{\bs{t}}^\alpha=N_{\bs{t}}\cdot N_{\bs{t}}^\alpha$, where
	\begin{equation*}
		N_{\bs{t}}=g_{\bs{t}}^\prime(0)^{-(2N-1)\tilde{b}}\prod_{j=1}^{2N} h_{\bs{t},j}^\prime(\xi_{t_j}^j)^b g_{\bs{t},j}^\prime(0)^{\tilde{b}}\times \exp\left(\frac{c}{2}\sum_{j=1}^{2N} \mu_{\bs{t}}^j+ab\sum_{j=1}^{2N}\int_0^{t_j}\sum_{k\ne j}\csc^2(\theta^j_{\bs{t}}-\theta^k_{\bs{t}})h_{\bs{t},j}^\prime(\xi_{t_j}^j)^2\ud t_j\right),
	\end{equation*}
	and
	\begin{equation*}
		N_{\bs{t}}^\alpha=g_{\bs{t}}^\prime(0)^{-\tilde{b}}\LG_\alpha(\theta^1_{\bs{t}},\ldots,\theta^{2N}_{\bs{t}})\times \exp\left(-ab\sum_{j=1}^{2N}\int_0^{t_j}\sum_{k\ne j}\csc^2(\theta^j_{\bs{t}}-\theta^k_{\bs{t}})h_{\bs{t},j}^\prime(\xi_{t_j}^j)^2\ud t_j\right).
	\end{equation*}
	It is proved in~\cite{HealeyLawlerNSidedRadialSLE} that $N_{\bs{t}}$ is a local martingale with respect to $\PP$ satisfying
	\begin{equation*}
		\ud N_{\bs{t}}=N_{\bs{t}}\sum_{j=1}^{2N}b\frac{h_{\bs{t},j}^{\prime\prime}(\xi_{t_j}^j)}{h_{\bs{t},j}^\prime(\xi_{t_j}^j)}\ud\xi_{t_j}^j. 
	\end{equation*}
	Denote by $\QQ$ the probability measure obtained by tilting $\PP$ by $N_{\bs{t}}$. By Girsanov's theorem, under $\QQ$, we have 
	\begin{equation*}
		\ud \theta_{\bs{t}}^j=h_{\bs{t},j}^\prime(\xi_{t_j}^j)\ud B_{t_j}^j+a\sum_{k\ne j}\cot(\theta^j_{\bs{t}}-\theta^k_{\bs{t}})h_{\bs{t},k}'(\xi_{t_k}^k)^2\ud t_k, \quad 1\leq j\leq 2N,
	\end{equation*}
	where $\{B^j\}_{1\leq j\leq 2N}$ are independent standard Brownian motions. 
	
	It remains to prove that $N_{\bs{t}}^\alpha$ is a local martingale with respect to $\mathbb{Q}$. From the PDE~\eqref{eqn::PDEcircle}, we have
	\begin{align}
		\frac{\ud \LG_\alpha(\theta^1_{\bs{t}},\ldots,\theta^{2N}_{\bs{t}})}{\LG_\alpha(\theta^1_{\bs{t}},\ldots,\theta^{2N}_{\bs{t}})}
		=&\sum_{j=1}^{2N}\left[(\partial_j\log\LG_\alpha)(\theta^1_{\bs{t}},\ldots,\theta^{2N}_{\bs{t}})\ud\theta_{\bs{t}}^j+\frac{1}{2}\frac{\partial_j^2\LG_\alpha(\theta^1_{\bs{t}},\ldots,\theta^{2N}_{\bs{t}})}{\LG_\alpha(\theta^1_{\bs{t}},\ldots,\theta^{2N}_{\bs{t}})}h_{\bs{t},j}^\prime(\xi_{t_j}^j)^2\ud t_j\right]\notag\\
		=&\sum_{j=1}^{2N}(\partial_j\log\LG_\alpha)(\theta^1_{\bs{t}},\ldots,\theta^{2N}_{\bs{t}})h_{\bs{t},j}^\prime(\xi_{t_j}^j)\ud B_{t_j}^j
		+2a\tilde{b}\sum_{j=1}^{2N}h_{\bs{t},j}^\prime(\xi_{t_j}^j)^2\ud t_j\notag\\
		&+ab\sum_{j=1}^{2N}\sum_{k\ne j}\csc^2(\theta^k_{\bs{t}}-\theta^j_{\bs{t}})h_{\bs{t},j}^\prime(\xi_{t_j}^j)^2\ud t_j.\label{eqn::Malpha_boldt_aux1}
	\end{align}
	Moreover, $\partial_{t_j} \log g_{\bs{t}}^\prime(0)=\partial_{t_j}\aleph(\bs{t})=2ah_{\bs{t},j}^\prime(\xi_{t_j}^j)^2$ for $1\leq j\leq 2N$ gives
	\begin{align}\label{eqn::Malpha_boldt_aux2}
		\frac{\ud g_{\bs{t}}^\prime(0)^{-\tilde{b}}}{g_{\bs{t}}^\prime(0)^{-\tilde{b}}}=-2a\tilde{b}\sum_{j=1}^{2N}h_{\bs{t},j}^\prime(\xi_{t_j}^j)^2\ud t_j.
	\end{align}
	Combining~\eqref{eqn::Malpha_boldt_aux1} and~\eqref{eqn::Malpha_boldt_aux2}, we see that $N_{\bs{t}}^\alpha$ is a local martingale with respect to $\mathbb{Q}$ satisfying
	\begin{equation*}
		\ud N_{\bs{t}}^\alpha=N_{\bs{t}}^\alpha\sum_{j=1}^{2N}(\partial_j\log\LG_\alpha)(\theta^1_{\bs{t}},\ldots,\theta^{2N}_{\bs{t}})h_{\bs{t},j}^\prime(\xi_{t_j}^j)\ud B_{t_j}^j.
	\end{equation*}
	This completes the proof.   
\end{proof}

\begin{proof}[Proof of Lemma~\ref{lem::RN_Palpha_P}]
	Denote by $\PP_{\alpha}$ the probability measure obtained by tilting $\PP$ by $M^{\alpha}_{\bs{t}}$ in~\eqref{Malpha_boldt}, then for $1\leq j\leq 2N$, 
	\begin{equation}\label{rhSLE_theta}
		\ud \theta^j_{\bs{t}}=h_{\bs{t},j}^\prime(\xi_{t_j}^j)\ud B_{t_j}^j+(\partial_j\log\LG_\alpha)(\theta^1_{\bs{t}},\ldots,\theta^{2N}_{\bs{t}})h_{\bs{t},j}^\prime(\xi_{t_j}^j)^2\ud t_j+a\sum_{k\ne j}\cot(\theta^j_{\bs{t}}-\theta^k_{\bs{t}})h_{\bs{t},k}^\prime(\xi_{t_k}^k)^2\ud t_k, 
	\end{equation}
	where $\{B^j\}_{1\leq j\leq 2N}$ are independent standard Brownian motions under $\PP_\alpha$. 
	We will argue that $\PP_{\alpha}$ is the same as $\PP_{\alpha}^{(\bs{\theta})}$. 
	
	On the one hand, by the domain Markov property of chordal $N$-$\SLE_\kappa$, $(\eta_1,\ldots,\eta_{2N})$ commute with each other in the following sense. 
	Suppose that $\{\eta_j([0,T_j])\}_{1\leq j\leq 2N}$ are disjoint. 
	Let $\tau_j$ be a stopping time before $T_j$ for $1\leq j\leq 2N$. 
	For each $\eta_j$, conditionally on $\bigcup_{k\ne j}\eta_k([0,\tau_k])$, if $g:\unitD\setminus\bigcup_{k\ne j}\eta_k([0,\tau_k])\to\unitD$ is conformal with $g(0)=0$ and $g^\prime(0)>0$, then the $g$-image of $\eta_j$ is a radial Loewner chain in $\unitD$ started from $g(x_j)$ with $a$-usual parameterization and with driving function $(\xi^j_t,t\geq 0)$ satisfying SDE~\eqref{eq::marginal} up to a time change. 
	
	On the other hand, from~\eqref{rhSLE_theta} we know that under the measure $\PP_{\alpha}$ and conditionally on $\bigcup_{k\ne j}\eta_{k}([0,\tau_k])$, $(g(\eta_j(t_j)), t_j\geq 0)$ is a radial Loewner chain in $\unitD$ started from $g(x_j)$ and driven by $(\theta_{t_j}^j,t_j\geq 0)$ satisfying the SDE
	\begin{equation*}
		\ud\theta_{t_j}^j=h_{\bs{t},j}^\prime(\xi_{t_j}^j)\ud B_{t_j}^j+(\partial_j\log\LG_\alpha)(\theta^1_{\bs{t}},\ldots,\theta^{2N}_{\bs{t}})h_{\bs{t},j}^\prime(\xi_{t_j}^j)^2\ud t_j,
	\end{equation*}
	where $\bs{t}=(\tau_1,\ldots,\tau_{j-1},t_j,\tau_{j+1},\ldots,\tau_{2N})$, $t_j\geq 0$. 
	This is the same as SDE~\eqref{eq::marginal} after a time change, which implies that the radial Loewner curves $(\eta_1,\ldots,\eta_{2N})$ under $\PP_\alpha$ locally commute with each other in the sense of the last paragraph. 
	Therefore, the joint law $\PP_{\alpha}^{(\bs{\theta})}$ of chordal $N$-$\SLE_\kappa$ $(\eta_1,\ldots,\eta_{2N})$ coincides with $\PP_\alpha$ up to any stopping time before the lifetime. 
	
	Finally, we parameterize these $2N$-tuple of curves with $a$-common parameterization by making the time change $\ud t_j=h_{\bs{t},j}'(\xi_{t_j}^j)^{-2}\ud t$ (see~\cite[Lemma~3.2]{HealeyLawlerNSidedRadialSLE}), and take the same notations as in Section~\ref{subsec::pre_2NradialSLE}, then the $2N$-time-parameter local martingale~\eqref{Malpha_boldt} becomes $M_t^{\alpha}$ in~\eqref{eqn::mart_alpha}.
	This completes the proof. 
	In addition, the SDE~\eqref{rhSLE_theta} becomes~\eqref{eqn::SDE_chordalNSLE} under $a$-common parameterization, which describes the driving function under $\PP_\alpha^{(\bs{\theta})}$.
\end{proof}

Now, we are ready to derive the transition densities and quasi-invariant densities for $(\bs{\theta}_t, 0\le t<T)$ under $\PP_{\alpha}^{(\bs{\theta})}$.

\begin{lemma}\label{lem::density_alpha}
	Under $\PP_{\alpha}^{(\bs{\theta})}$, $(\bs{\theta}_t, 0\leq t<T)$ is a Markov process with transition density
	\begin{equation}\label{eqn::density_alpha}
		p_\alpha(t;\bs{\theta},\bs{\theta}^\prime)=p_*(t;\bs{\theta},\bs{\theta}^\prime)e^{-\armexp\cdot 4aNt}\frac{G_\alpha(\bs{\theta})}{G_\alpha(\bs{\theta}^\prime)},
	\end{equation}
	where $\armexp$ is the $2N$-arm exponent defined in~\eqref{eqn::armexponent} and $G_{\alpha}=\LG_{*}/\LG_{\alpha}$ is the Green's function in Definition~\ref{def::NSLEGreen}. 
\end{lemma}

\begin{proof}
	For each $t>0$, from~\eqref{eqn::mart_alpha}, Lemma~\ref{lem::RN_Palpha_P} and~\eqref{eqn::mart*}, we obtain the Radon-Nikodym derivative of $\PP_\alpha^{(\bs{\theta})}$ against $\PP_*^{(\bs{\theta})}$ when both measures are restricted to $\LF_t$ and $\{T>t\}$: 
	\begin{equation}\label{eqn::RN_Palpha_P*}
		\frac{\ud\PP_{\alpha}^{(\bs{\theta})}[\cdot|_{\LF_t\cap\{T>t\}}]}{\ud\PP_*^{(\bs{\theta})}[\cdot|_{\LF_t}]}
		=\frac{M^{\alpha}_t/M^{\alpha}_0}{M^*_t/M^*_0}
		=g_t^\prime(0)^{-\armexp}\frac{\LG_\alpha(\bs{\theta}_t)/\LG_\alpha(\bs{\theta})}{\LG_*(\bs{\theta}_t)/\LG_*(\bs{\theta})}=e^{-\armexp\cdot 4aNt} \frac{G_\alpha(\bs{\theta})}{G_\alpha(\bs{\theta}_t)},
	\end{equation}
	where 
	\begin{equation*}
		\armexp=\frac{(4N^2-1)a}{4}+\tilde{b}=\frac{16N^2-(4-\kappa)^2}{8\kappa}
	\end{equation*}
	is the $2N$-arm exponent. Since $(\bs{\theta}_t,t\geq 0)$ is a Markov process with transition density $p_*(t;\bs{\theta},\bs{\theta}')$ under $\PP_*^{(\bs{\theta})}$, the Radon-Nikodym derivative gives the transition density $p_{\alpha}(t;\bs{\theta},\bs{\theta}^\prime)$.
\end{proof}

\begin{lemma}
	The constant  
	\begin{equation}\label{eq::LJ_alpha}
		\LJ_\alpha=\int_{\mathcal{X}_{2N}}f_{4a}(\bs{\theta})G_\alpha(\bs{\theta})^{-1}\ud\bs{\theta}
	\end{equation}
	is finite, where $f_{4a}$ is defined in~\eqref{eqn::def_F_f} and $G_{\alpha}$ is Green's function in Definition~\ref{def::NSLEGreen}. We define 
	\begin{equation}\label{eqn::quasi_inv_density_def}
		p_{\alpha}(\bs{\theta})=\LJ_{\alpha}^{-1}f_{4a}(\bs{\theta})G_\alpha(\bs{\theta})^{-1}. 
	\end{equation}
	Then $p_{\alpha}(\bs{\theta})$ is a quasi-invariant density for $(\bs{\theta}_t, 0\leq t<T)$ under $\PP_\alpha^{(\bs{\theta})}$ in the sense that
	\begin{equation}\label{eqn::quasi_inv_density_property}
		\int_{\mathcal{X}_{2N}}p_\alpha(\bs{\theta}) p_\alpha(t;\bs{\theta},\bs{\theta}')\ud \bs{\theta}
		=p_\alpha(\bs{\theta}')e^{-\armexp\cdot 4aNt}. 
	\end{equation}
	Moreover, for any $\bs{\theta},\bs{\theta}^\prime\in \mathcal{X}_{2N}$, we have
	\begin{equation}\label{eqn::quasi_inv_density_cvg}
		p_\alpha(t;\bs{\theta},\bs{\theta}^\prime)=\LJ_\alpha G_\alpha(\bs{\theta})e^{-\armexp\cdot 4aNt} p_\alpha(\bs{\theta}^\prime) (1+O(e^{-\cvgexp\cdot 4aNt})).
	\end{equation}
\end{lemma}
\begin{proof}
	Note that 
	\begin{equation*}
		f_{4a}(\bs{\theta})G_\alpha(\bs{\theta})^{-1}
		=\mathcal{I}_{4a}^{-1}F_{4a}(\bs{\theta})\LG_\alpha(\bs{\theta})\LG_*(\bs{\theta})^{-1}
		=\mathcal{I}_{4a}^{-1}F_{3a}(\bs{\theta})\LG_\alpha(\bs{\theta}).
	\end{equation*}
	From the power-law bound~\eqref{eq::G_PLB} and the relation $2b=3a-1$, we know that $F_{3a}(\bs{\theta})\LG_\alpha(\bs{\theta})$ is bounded on $\mathcal{X}_{2N}$. Thus $\LJ_{\alpha}$ is a finite constant.
	The relation~\eqref{eqn::quasi_inv_density_property} follows from Lemma~\ref{lem::density*} and Lemma~\ref{lem::density_alpha}. 
	Finally, we have
	\begin{align*}
		p_\alpha(t;\bs{\theta},\bs{\theta}^\prime)
		&=e^{-\armexp\cdot 4aNt}p_*(t;\bs{\theta},\bs{\theta}^\prime)\frac{G_\alpha(\bs{\theta})}{G_\alpha(\bs{\theta}^\prime)}\tag{due to~\eqref{eqn::density_alpha}}\\
		&=e^{-\armexp\cdot 4aNt} f_{4a}(\bs{\theta}^\prime)(1+O(e^{-\cvgexp\cdot 4aNt}))\frac{G_\alpha(\bs{\theta})}{G_\alpha(\bs{\theta}^\prime)}\tag{due to~\eqref{eqn::density*_asy}}\\
		&=\LJ_\alpha G_\alpha(\bs{\theta})e^{-\armexp\cdot 4aNt}p_\alpha(\bs{\theta}^\prime)(1+O(e^{-\cvgexp\cdot 4aNt}))\tag{due to~\eqref{eqn::quasi_inv_density_def}},
	\end{align*}
	which proves~\eqref{eqn::quasi_inv_density_cvg}.
\end{proof}

\subsection{Proof of Proposition~\ref{prop::multiplechordalestimates}, Theorem~\ref{thm::NChordalSLEGreenFunction} and Corollary~\ref{cor::chordaltoradial}}\label{subsec::proofs}

\begin{proof}[Proof of Proposition~\ref{prop::multiplechordalestimates}]
	For each $\bs{\theta}=(\theta^1,\ldots,\theta^{2N})\in\mathcal{X}_{2N}$, we have
	\begin{align*}
		\PP_{\alpha}^{(\bs{\theta})}[T>t]
		&=\int_{\mathcal{X}_{2N}}p_\alpha(t;\bs{\theta},\bs{\theta}^\prime)\ud\bs{\theta}^\prime\tag{due to~\eqref{eq::Tetsimate}}\\
		&=\int_{\mathcal{X}_{2N}}e^{-\armexp\cdot 4aNt}p_*(t;\bs{\theta},\bs{\theta}^\prime)\frac{G_\alpha(\bs{\theta})}{G_\alpha(\bs{\theta}^\prime)}\ud \bs{\theta}^\prime \tag{due to~\eqref{eqn::density_alpha}}\\
		&=\int_{\mathcal{X}_{2N}}e^{-\armexp\cdot 4aNt} f_{4a}(\bs{\theta}^\prime)(1+O(e^{-\cvgexp\cdot 4aNt}))\frac{G_\alpha(\bs{\theta})}{G_\alpha(\bs{\theta}^\prime)}\ud \bs{\theta}^\prime\tag{due to~\eqref{eqn::density*_asy}}\\
		&=\LJ_\alpha G_\alpha(\bs{\theta})e^{-\armexp\cdot 4aNt} (1+O(e^{-\cvgexp\cdot 4aNt})),
	\end{align*}
	where $\LJ_{\alpha}$ is defined in~\eqref{eq::LJ_alpha}. This completes the proof. 
\end{proof}

\begin{proof}[Proof of Theorem~\ref{thm::NChordalSLEGreenFunction}]
	For $r>0$, we denote
	\begin{equation*}
		P(\bs{\theta};r):=\PP_\alpha^{(\bs{\theta})}[\dist(0,\gamma_j)< r, 1\leq j\leq N],
	\end{equation*} 
	then it suffices to prove the asmyptotic
	\begin{equation}\label{eqn::NChordalSLEGreen_ASY}
		P(\bs{\theta};r)=C\cdot G_\alpha(\bs{\theta})r^{\armexp}(1+O(r^{u})), \quad \text{as }r\to 0+,
	\end{equation}
	where $C\in (0,\infty)$ is a constant depending on $\kappa, N, \alpha$ and $u>0$ is a constant depending on $\kappa, N$. To prove~\eqref{eqn::NChordalSLEGreen_ASY}, on the one hand, we will  compare this probability with $\PP_\alpha^{(\bs{\theta})}[T>t]$ for proper $t$, whose asymptotic is known due to  Proposition~\ref{prop::multiplechordalestimates}. On the other hand, we define 
	\begin{equation}\label{eqn::def_Pr}
		P(r):=\int_{\mathcal{X}_{2N}}p_\alpha(\bs{\theta})P(\bs{\theta};r)\ud\bs{\theta},
	\end{equation}
	where $p_\alpha(\bs{\theta})$ is the quasi-invariant density for $(\bs{\theta}_t,0\leq t<T)$ under $\PP_\alpha^{(\bs{\theta})}$ defined in~\eqref{eqn::quasi_inv_density_def}, and we will prove the asymptotic 
	\begin{equation}\label{eqn::Asymptotic_Pr}
		P(r)=Lr^{\armexp}(1+O(r)), \quad \text{as }r\to 0+, 
	\end{equation}
	for some constant $L>0$. Combining~\eqref{eqn::Asymptotic_Pr} and~\eqref{eqn::quasi_inv_density_cvg}, we will obtain~\eqref{eqn::NChordalSLEGreen_ASY}.

	First, to compare the events $\{\text{dist}(0,\gamma_j)<r, 1\leq j\leq N\}$ and $\{T>t\}$, we regard $(\gamma_1,\ldots,\gamma_N)$ as a $2N$-tuple of curves $(\eta_1,\ldots,\eta_{2N})$ as described in Section~\ref{subsec::intro_Green} and derive the relation between different time parameterizations of the curves. Suppose that for each $1\leq j\leq 2N$, $\eta_j$ is parameterized with $a$-usual parameterization. Then we get a $2N$-time parameter $\bs{t}=(t_1,\ldots,t_{2N})$. Recall the normalized conformal maps $g_{t_j}^j$, $g_{\bs{t}}$, $g_{\bs{t},j}$ and the corresponding covering conformal maps defined in Section~\ref{subsec::TransitionDensity}. A standard calculation as in~\cite[Section~3.1]{ZhanGreen2SLE} gives that
	\begin{equation*}
		\ud\log g_{\bs{t}}^\prime(0)=\ud\aleph(\bs{t})=2a\sum_{j=1}^{2N}h_{\bs{t},j}^\prime(\xi_{t_j}^j)^2\ud t_j.
	\end{equation*}
	For each $1\leq j\leq 2N$, we have  $h_{\bs{t},j}^\prime(\xi_{t_j}^j)=|g_{\bs{t},j}^\prime(e^{2\ii\xi_{t_j}^j})|\in(0,1]$, and for $\bs{t}=(0,\ldots,0,t_j,0,\ldots,0)$, we have $\aleph(\bs{t})=2at_j$. Therefore,  for $\bs{t}=(t_1, \ldots, t_{2N})$, 
	\begin{equation*}
		2a\max_{1\leq j\leq 2N}t_j\leq \aleph(\bs{t})\leq 2a(t_1+\cdots+t_{2N}).
	\end{equation*}
	From~\cite[Lemma~3.2]{HealeyLawlerNSidedRadialSLE}, there exists a time change $\bs{t}=\bs{t}(t)=(t_1(t),\ldots,t_{2N}(t))$ with $t_j^\prime(t)=h_{\bs{t},j}^\prime(\xi_{t_j}^j)^{-2}$ for $1\leq j\leq 2N$, such that $(\bs{\eta}(t)=(\eta_1(t_1(t)),\ldots,\eta_{2N}(t_{2N}(t))),0\leq t<T)$ is parameterized with $a$-common parameterization. In this case, we have $\aleph(t)=4aNt$, which gives the relation 
	\begin{equation*}
		\max_{1\leq j\leq 2N}t_j\leq 2Nt\leq t_1+\cdots+t_{2N},
	\end{equation*}
	or we may write it as 
	\begin{equation}\label{eqn::time_relation}
		t\leq t_j(t)\leq 2Nt, \quad \text{for all }1\leq j\leq 2N.
	\end{equation}

	Second, we prove the relation
	\begin{equation}\label{eqn::LHS}
		\{\dist(0,\gamma_j)< r, 1\leq j\leq N\}\subset\left\{T>-\frac{1}{4aN}\log(4r)\right\}.
	\end{equation}
	Suppose that for each $1\leq j\leq 2N$, $\eta_j$ is parameterized with $a$-usual parameterization, and we define $\tau_j=\min(\inf\{t_j:\dist(0,\eta_j(0,t_j])\leq r\}, t_j(T))$. On the event $\{\dist(0,\gamma_j)< r, 1\leq j\leq N\}$, for each $1\leq j\leq N$, we have $\tau_{a_j}\wedge\tau_{b_j}<\infty$, and by Koebe's $1/4$ Theorem, $\tau_{a_j}\wedge\tau_{b_j}>-\frac{1}{2a}\log(4r)$. 
	Note that $\{\eta_j(0,\tau_j),1\leq j\leq 2N\}$ are still disjoint, then by~\eqref{eqn::time_relation}, the lifetime $T$ of $(\eta_1,\ldots,\eta_{2N})$ has a lower bound
	\begin{equation*}
		T>\frac{1}{2N}\max_{1\leq j\leq 2N}\tau_j>-\frac{1}{4aN}\log(4r).
	\end{equation*}
	This completes the proof of~\eqref{eqn::LHS}.

	Third, we will derive the asymptotic~\eqref{eqn::Asymptotic_Pr}. Fix some $s\in[0,-\frac{1}{4aN}\log(4r)]$. On the one hand, we have $\{\dist(0,\gamma_j)<r, 1\leq j\leq N\}\subset\{T>s\}$ from~\eqref{eqn::LHS}. On the other hand, when $T>s$ happens, by Koebe's $1/4$ Theorem and~\eqref{eqn::time_relation}, we have 
	\begin{equation*}
		\dist(0,\eta_j(0,s])\geq \frac{1}{4}e^{-2at_j(s)}\geq \frac{1}{4}e^{-4aNs}\geq r, \quad \text{for all }1\leq j\leq 2N.
	\end{equation*}
	Therefore, the event $\{\dist(0,\gamma_j)< r, 1\leq j\leq N\}$ 
	happens if and only if $\{T>s\}$ holds, and the event  $\{\text{dist}(0,\gamma_j\setminus(\eta_{a_j}(0,s]\bigcup\eta_{b_j}(0,s]))<r, 1\leq j\leq N\}$ happens. Suppose that $T>s$ happens and let $0<R_1<R_2<1$ be such that
	\begin{equation}\label{eqn::relation_R}
		\frac{e^{-4aNs}R_1}{(1-R_1)^2}=\frac{e^{-4aNs}R_2}{(1+R_2)^2}=r.
	\end{equation}
	By Koebe's distortion Theorem, we have
	\begin{equation*}
		\{|z|<R_1\}\subset g_s(\{|z|<r\})\subset\{|z|<R_2\}.
	\end{equation*}
	Then the domain Markov property for chordal $N$-$\SLE_\kappa$ gives that, 
	\begin{equation}\label{eq::DMPbound}
		P(\bs{\theta}_s;R_1)\leq \PP_\alpha^{(\bs{\theta})}[\dist(0,\gamma_j)< r, 1\leq j\leq N|\mathcal{F}_s,T>s]\leq P(\bs{\theta}_s;R_2).
	\end{equation}
	From the definition~\eqref{eqn::def_Pr} of $P(r)$ and the property~\eqref{eqn::quasi_inv_density_property} of quasi-invariant density $p_\alpha$, we obtain
	\begin{equation}\label{eqn::bound_Pr}
		e^{-\armexp\cdot 4aNs} P(R_1)\leq P(r)\leq e^{-\armexp\cdot 4aNs}P(R_2).
	\end{equation}
	To find the asymptotic of $P(r)$ as $r\to 0+$, we define for $r>0$, 
	\begin{equation*}
		q(r)=r^{-\armexp}P(r). 
	\end{equation*}
	Suppose $r,R\in(0,1)$ with $r<\frac{R}{(1+R)^2}$. Then we have the following estimates: 
	\begin{itemize}
		\item The upper bound: by choosing $s>0$ such that $e^{4aNs}=\frac{R/r}{(1+R)^2}$, we conclude from~\eqref{eqn::bound_Pr} that
		\begin{equation*}
			P(r)\leq e^{-\armexp\cdot 4aNs}P(R)=\left(\frac{R/r}{(1+R)^2}\right)^{-\armexp}P(R), \text{ i.e., } q(r)\leq (1+R)^{2\armexp}q(R).
		\end{equation*}
		\item The lower bound: by choosing $s>0$ such that $e^{4aNs}=\frac{R/r}{(1-R)^2}$, we conclude from~\eqref{eqn::bound_Pr} that
		\begin{equation*}
			P(r)\geq e^{-\armexp\cdot 4aNs}P(R)=\left(\frac{R/r}{(1-R)^2}\right)^{-\armexp}P(R), \text{ i.e., } q(r)\geq (1-R)^{2\armexp}q(R).
		\end{equation*}
	\end{itemize}
	Combining the upper and lower bounds, we obtain 
	\begin{equation}\label{eqn::q_estimate}
		(1-R)^{2\armexp}q(R)\leq q(r)\leq (1+R)^{2\armexp}q(R),\quad \text{for }r<\frac{R}{(1+R)^2}.
	\end{equation}
	Thus, $\lim_{r\to 0+}\log q(r)$ converges to a finite number, which implies that $\lim_{r\to 0+}q(r)$ converges to a finite positive number, denoted as $L$. Fixing $R\in(0,1)$ and sending $r\to 0+$ in~\eqref{eqn::q_estimate}, we get
	\begin{equation*}
		L(1+R)^{-2\armexp}\leq q(R)\leq L(1-R)^{-2\armexp}.
	\end{equation*}
	which implies the asymptotic $q(r)=L(1+O(r))$ as $r\to 0+$, and proves~\eqref{eqn::Asymptotic_Pr}.

	Finally, we come to the proof of the asymptotic~\eqref{eqn::NChordalSLEGreen_ASY}. From the relation~\eqref{eqn::relation_R}, for $j=1,2$, we have  $R_j=e^{4aNs}r(1+O(e^{4aNs}r))$ as $r\to 0+$. Then, we obtain
	\begin{align*}
		P(\bs{\theta};r)
		=&\PP_\alpha^{(\bs{\theta})}[T>s]\times P(e^{4aNs}r(1+O(e^{4aNs}r)))\tag{due to~\eqref{eqn::quasi_inv_density_cvg} and~\eqref{eq::DMPbound}}\\
		=&\LJ_\alpha G_\alpha(\bs{\theta})e^{-\armexp\cdot 4aNs}(1+O(e^{-\cvgexp\cdot 4aNs}))\\
		&\times L(e^{4aNs}r(1+O(e^{4aNs}r)))^{\armexp}(1+O(e^{4aNs}r))\tag{due to~\eqref{eqn::multiplechordalestimates} and~\eqref{eqn::Asymptotic_Pr}}\\
		=&\LJ_\alpha L G_\alpha(\bs{\theta})r^{A_{2N}}(1+O(e^{-\cvgexp\cdot 4aNs})+O(e^{4aNs}r)), \quad \text{as }r\to 0+.
	\end{align*}
	Taking $C=\LJ_\alpha L>0$, $u=\frac{\cvgexp}{1+\cvgexp}>0$ and $s>0$ such that $e^{-(1+\cvgexp)4aNs}=r$, we obtain~\eqref{eqn::NChordalSLEGreen_ASY} and complete the proof. 
\end{proof}

\begin{proof}[Proof of Corollary~\ref{cor::chordaltoradial}]
	Fix $s>0$. 
	We first show the convergence in~\eqref{eqn::TV_Palpha_P*}. 
	For $p>0$, we define 
	\[E_p=\{\eta_1,\ldots,\eta_{2N} \text{ all hit } e^{-p}\unitD\}. \]
	For $p>0$ large enough (such that $\text{dist}(0,\eta_j(0,s])\geq \frac{1}{4}e^{-4aNs}>e^{-p}$), from the domain Markov property of chordal $N$-$\SLE_\kappa$, we have 
	\begin{equation*}
		\frac{\ud\PP_{\alpha\cond p}^{(\bs{\theta})}\left[\cdot|_{\LF_s}\right]}{\ud\PP_\alpha^{(\bs{\theta})}\left[\cdot|_{\LF_s\cap \{T>s\}}\right]}
		=\frac{\PP_\alpha^{(\bs{\theta}_s)}[\tilde{E}_p]}{\PP_\alpha^{(\bs{\theta})}[E_p]},
		\quad\text{where }\tilde{E}_p=\{\eta_1,\ldots,\eta_{2N} \text{ all hit } g_s(e^{-p}\unitD)\}. 
	\end{equation*}
	From the similar analysis after~\eqref{eqn::relation_R}, we know that by Koebe's distortion Theorem, there exist $p_1,p_2>0$ with $e^{-p_j}=e^{-p+4aNs}(1+O(e^{-p}))$ such that $e^{-p_1}\unitD\subset g_s(e^{-p}\unitD)\subset e^{-p_2}\unitD$. Then we have
	\begin{equation*}
		\PP_\alpha^{(\bs{\theta}_s)}[E_{p_1}]\leq \PP_\alpha^{(\bs{\theta}_s)}[\tilde{E}_p]\leq \PP_\alpha^{(\bs{\theta}_s)}[E_{p_2}],
	\end{equation*}
	and by~\eqref{eqn::NChordalSLEGreen_ASY}, 
	\begin{equation*}
		\PP_\alpha^{(\bs{\theta}_s)}[\tilde{E}_p]=C\cdot G_\alpha(\theta^1_s,\ldots,\theta^{2N}_s)e^{\armexp\cdot 4aNs}e^{-p\armexp}(1+O(e^{-up})).
	\end{equation*}
	Therefore, we obtain
	\begin{equation*}
		\frac{\ud\PP_{\alpha\cond p}^{(\bs{\theta})}\left[\cdot|_{\LF_s}\right]}{\ud\PP_\alpha^{(\bs{\theta})}\left[\cdot|_{\LF_s\cap \{T>s\}}\right]}
		=\frac{\PP_\alpha^{(\bs{\theta}_s)}[\tilde{E}_p]}{\PP_\alpha^{(\bs{\theta})}[E_p]}
		=e^{\armexp\cdot 4aNs}\frac{G_\alpha(\theta^1_s,\ldots,\theta^{2N}_s)}{G_\alpha(\theta^1,\ldots,\theta^{2N})}(1+O(e^{-up})).
	\end{equation*}
	Combining with~\eqref{eqn::RN_Palpha_P*}, we conclude that 
	\begin{equation*}
		\frac{\ud\PP_{\alpha\cond p}^{(\bs{\theta})}\left[\cdot|_{\LF_s}\right]}{\ud\PP_*^{(\bs{\theta})}\left[\cdot|_{\LF_s}\right]}
		=1+O(e^{-up})\to 1 \text{ uniformly as } p\to \infty,
	\end{equation*}
	which implies that
	\begin{equation*}
		\lim_{p\to\infty}\mathbb{E}_*^{(\bs{\theta})}\left[\left(1-\frac{\ud \PP_{\alpha\cond p}^{(\bs{\theta})}\left[\cdot|_{\LF_s}\right]}{\ud \PP_*^{(\bs{\theta})}\left[\cdot|_{\LF_s}\right]}\right)^+\right]=0. 
	\end{equation*}
	This gives the convergence in total-variation distance~\eqref{eqn::TV_Palpha_P*}. 
	
	Next, we show the convergence in~\eqref{eqn::dist_TV_nu}. 	
	From the end of the proof of Lemma~\ref{lem::RN_Palpha_P}, we know that the driving function $(\bs{\theta}_t,t\geq 0)$ under $\PP_\alpha^{(\bs{\theta})}$ with $a$-common parameterization satisfies SDE~\eqref{eqn::SDE_chordalNSLE}. 
	We denote by $\PP_{\alpha\cond t}^{(\bs{\theta})}$ the measure $\PP_\alpha^{(\bs{\theta})}$ conditional on $\{T>t\}$. 
	Since $\nu_{\alpha\cond t}^{(\bs{\theta})}$ is the induced measure on the driving function $(\bs{\theta}_t,t\geq 0)$ under $\PP_{\alpha\cond t}^{(\bs{x})}$ and $\nu_*^{(\bs{\theta})}$ is the induced measure on the driving function under $\PP_*^{(\bs{\theta})}$, it remains to prove that for any $s>0$, 
	\begin{equation}\label{eqn::dist_TV_PP}
		\lim_{t\to\infty}\dist_{TV}\left(\PP_{\alpha\cond t}^{(\bs{\theta})}[\cdot\cond_{\LF_s}],\PP_*^{(\bs{\theta})}[\cdot\cond_{\LF_s}]\right)=0.
	\end{equation}
	For any $t>s>0$, from the domain Markov property of chordal $N$-$\SLE_\kappa$ and Proposition~\ref{prop::multiplechordalestimates}, we have \begin{equation*}
		\frac{\ud\PP_{\alpha\cond t}^{(\bs{\theta})}\left[\cdot|_{\LF_s}\right]}{\ud\PP_\alpha^{(\bs{\theta})}\left[\cdot|_{\LF_s\cap\{T>s\}}\right]}
		=\frac{\PP_\alpha^{(\bs{\theta}_s)}[T>t-s]}{\PP_\alpha^{(\bs{\theta})}[T>t]}
		=e^{\armexp\cdot 4aNs}\frac{G_\alpha(\theta^1_s,\ldots,\theta^{2N}_s)}{G_\alpha(\theta^1,\ldots,\theta^{2N})}(1+O(e^{-\cvgexp\cdot 4aNt})).
	\end{equation*}
	Combining with~\eqref{eqn::RN_Palpha_P*}, we conclude that 
	\begin{equation*}
		\frac{\ud\PP_{\alpha\cond t}^{(\bs{\theta})}\left[\cdot|_{\LF_s}\right]}{\ud\PP_*^{(\bs{\theta})}\left[\cdot|_{\LF_s}\right]}
		=1+O(e^{-\cvgexp\cdot 4aNt})\to 1 \text{ uniformly as } t\to \infty.
	\end{equation*}
	This gives the convergence in total-variation distance~\eqref{eqn::dist_TV_PP} thus~\eqref{eqn::dist_TV_nu}, and completes the proof.
\end{proof}

\section{Ising model} \label{sec::Ising_model}
In this section, we will consider convergence of Ising interfaces in polygons and in annulus. 
To this end, we first give the notion of convergence of polygons and convergence of annulus. 

\paragraph{Convergence of polygons}
Let $l\geq 1$. A \textit{discrete polygon}  is a finite graph $\Omega$ of $\mathbb{Z}^2$ with $l$ marked points $x_1,\ldots,x_{l}$ on its boundary, whose precise definition is given as follows. Let $x_1,\ldots,x_{l}\in V(\mathbb{Z}^2)$ be $l$ distinct vertices. Let $(x_1x_{2}),\ldots, (x_{l}x_1)$ be $l$ paths in $\mathbb{Z}^2$ satisfying the following 4 conditions:
(1): all the paths are edge-avoiding and $(x_{j-1}^{}x_j^{})\cap (x_{j}^{}x_{j+1}^{})=\{x_j^{}\}$ for  $1\le j\le l$;
(2): if $j\notin \{k-1,k+1\}$, then $(x_{k}^{}x_{k+1}^{})\cap (x_{j}^{}x_{j+1}^{})=\emptyset$;
(3): the loop obtained by concatenating all $(x_{j}^{}x_{j+1}^{})$ is a Jordan curve;
(4): the infinite connected component of $ \mathbb{Z}^2\setminus (\cup_j (x_{j}^{}x_{j+1}^{}))$ is on the right of the path $(x_1^{}x_{2}^{})$. 
Given $\{(x_j^{}x_{j+1}^{})\}_{1\leq j\leq l}$, the discrete polygon $(\Omega^{};x_1^{},\ldots,x_{l}^{})$ is defined in the following way: $\Omega^{}$ is a subgraph of $\mathbb{Z}^2$ induced  by the vertices lying on or enclosed by the loop obtained by concatenating all $(x_j x_{j+1})$. The (discrete) dual annular polygon $(\Omega^{\bullet};x_1^{\bullet},\ldots,x_{l}^{\bullet})$ is defined via the primal annular polygon as follows. The dual graph $\Omega^{\bullet}$ of $\Omega^{}$ is defined to be a subgraph of $(\mathbb{Z}^2)^{\bullet}$ with edge set consisting of edges lying on the inner of the planar domain corresponding to $\Omega^{}$ and vertex set consisting of endpoints of these edges. For $1\leq j\leq l$, let $x_j^{\bullet}$ be a dual vertex on the boundary of $\Omega^{\bullet}$ that  is nearest to $x_{j}^{}$. 

	Let $\{\Omega^{\delta}\}_{\delta>0}$ and $\Omega$ be planar simply connected domains $\Omega^{\delta}, \Omega\subsetneq\C$, all containing a common point $u$. We say that $\Omega^{\delta}$ converges to $\Omega$ in the sense of \textit{kernel convergence} with respect to $u$,
	denoted $\Omega^{\delta}\to\Omega$, if the following 2 conditions are satisfied: (1): 
	every $z\in\Omega$ has some neighborhood $U_z$ such that $U_z\subset\Omega^{\delta}$, for all small enough $\delta > 0$; and 
	(2): for every point $p\in\partial\Omega$, there exists a sequence $p^{\delta}\in\partial\Omega^{\delta}$ such that $p^{\delta}\to p$. 
	If $\Omega^{\delta}\to\Omega$ in the sense of kernel convergence with respect to $u$, the same convergence holds with respect to any $\tilde{u}\in\Omega$. 
	We say that $\Omega^{\delta}\to\Omega$ in the \textit{Carath\'{e}odory sense} as $\delta\to 0$.  
	Note that $\Omega^{\delta}\to \Omega$ in the Carath\'{e}odory sense if and only if there exist conformal maps $\varphi_{\delta}$ from $\unitD$ onto $\Omega^{\delta}$, and a conformal map $\varphi$ from $\unitD$ onto $\Omega$, such that $\varphi_{\delta}\to\varphi$ locally uniformly on $\unitD$ as $\delta\to 0$.
	For {polygons}, we say that a sequence of discrete polygons $(\Omega^{\delta}; x_1^{\delta}, \ldots, x_{l}^{\delta})$
	converges as $\delta \to 0$ to an polygon $(\Omega; x_1, \ldots, x_{l})$ in the \textit{Carath\'{e}odory sense}
	if there exist conformal maps $\varphi_{\delta}$ from $\unitD$ onto $\Omega^{\delta}$,
	and a conformal map $\varphi$ from $\unitD$ onto $\Omega$,
	such that $\varphi_{\delta} \to \varphi$ locally uniformly on $\unitD$,
	and $\varphi_{\delta}^{-1}(x_j^{\delta}) \to \varphi^{-1}(x_j)$ for all $1\le j\le l$. 
	
	Note that Carath\'{e}odory convergence allows wild behavior of the boundaries around the marked points. 
	In order to ensure 
	precompactness of the interfaces in Proposition~\ref{prop::Ising_polygon}, we need a convergence on polygons stronger than the above Carath\'{e}odory convergence. 
	We say that a sequence of discrete polygons $(\Omega^{\delta}; x_1^{\delta}, \ldots, x_{l}^{\delta})$  
	converges as $\delta \to 0$ to an polygon $(\Omega; x_1, \ldots, x_{l})$ in the \textit{close-Carath\'{e}odory sense} if it converges in the Carath\'{e}odory sense, and in addition, for each $j\in\{1, \ldots, l\}$, we have $x_j^{\delta}\to x_j$ as $\delta\to 0$ and the following is fulfilled: 
	Given a reference point $u\in\Omega$ and 
	$r>0$ small enough, let $S_r$ be the arc of $\partial B(x_j,r)\cap\Omega$ disconnecting (in $\Omega$) $x_j$ from $u$ and from all other arcs of this set. We require that, for each $r$ small enough and for all sufficiently small $\delta$ (depending on $r$), the boundary point $x_j^{\delta}$ is connected to the midpoint of $S_r$ inside $\Omega^{\delta}\cap B(x_j,r)$. This notion was introduced by Karrila in~\cite{KarrilaConformalImage}, see in particular~\cite[Theorem~4.2]{KarrilaConformalImage}. (See also~\cite{KarrilaMultipleSLELocalGlobal} and~\cite{ChelkakWanMassiveLERW}.) 
	
Later, we will consider Ising interfaces in polygons with alternating $\oplus/\ominus$ boundary conditions~\eqref{eqn::bound_condi_polygon} in Section~\ref{sec::Ising_polygon}; and we will consider Ising interfaces in polygons with alternating $\oplus/\ominus/\free$ boundary conditions~\eqref{eqn::bounda_condi_one_free} in Section~\ref{sec::local_one_free}. 
	
\paragraph{Convergence of annular polygons}
A \textit{discrete annular polygon} 
	is a finite doubly connected subgraph $\mathcal{A}$ of $ \mathbb{Z}^2$ (that is, a finite connected subgraph of $\mathbb{Z}^2$ whose complement in $\mathbb{Z}^2$ has two connected components) with $2N$ marked points $x_1^{}, x_{2}^{},\ldots, x_{2N}^{}$ on the outer boundary in counterclockwise order, whose precise definition is given as follows. 
	Let $x_{1}^{},\ldots, x_{2N}^{}\in V(\mathbb{Z}^2)$ be $2N$ distinct vertices. Let $(x_1^{}x_{2}^{}), \ldots , (x_{2N}^{}x_{1}^{})$ be $2N$ paths in $\mathbb{Z}^2$ satisfying the following 5 conditions:
	(1): all the paths are edge-avoiding and $(x_{j-1}^{}x_j^{})\cap (x_{j}^{}x_{j+1}^{})=\{x_j^{}\}$ for  $1\le j\le 2N$, where $x_0:=x_{2N}$ and $x_{2N+1}:=x_{1}$ by convention;
	(2): if $j\notin \{k-1,k+1\}$, then $(x_{k}^{}x_{k+1}^{})\cap (x_{j}^{}x_{j+1}^{})=\emptyset$;
	(3): the loop obtained by concatenating all $(x_{j}^{}x_{j+1}^{})$ is a Jordan curve, denote it by $\outb^{}$ and call it the \textit{outer boundary} of $\mathcal{A}$; 
	(4): the infinite connected component of $ \mathbb{Z}^2\setminus (\cup_j (x_{j}^{}x_{j+1}^{}))$ is on the right of the path $(x_1^{}x_{2}^{})$;
	(5): there exists a Jordan curve $\intb^{}$ consisting of edges on $ \mathbb{Z}^2$ such that $\intb^{}$ is contained in the bounded planar domain enclosed by $\outb^{}$ and that there is no edge in $ \mathbb{Z}^2$ connecting $\intb^{}$ to $\outb^{}$, and we call $\intb$ the \textit{inner boundary} of $\mathcal{A}$. 
	Given $\{(x_j^{}x_{j+1}^{})\}_{1\leq j\leq 2N}$, the discrete annular polygon $(\mathcal{A}^{};x_1^{},\ldots,x_{2N}^{})$ is defined in the following way: $\mathcal{A}^{}$ is a subgraph of $\mathbb{Z}^2$ induced by vertices lying on or enclosed by the two disjoint Jordan curves $\outb^{}$ and $\intb^{}$.
	The (discrete) dual annular polygon $(\mathcal{A}^{\bullet};x_1^{\bullet},\ldots,x_{2N}^{\bullet})$ is defined via the primal annular polygon as follows. The dual graph $\mathcal{A}^{\bullet}$ of $\mathcal{A}^{}$ is defined to be a subgraph of $(\mathbb{Z}^2)^{\bullet}$ with edge set consisting of edges lying on the inner of the planar domain corresponding to $\mathcal{A}^{}$ and vertex set consisting of endpoints of these edges. Then $\mathcal{A}^{\bullet}$ is an annular subgraph of $(\mathbb{Z}^2)^{\bullet}$, and we denote by $\outb^{\bullet}$, $\intb^{\bullet}$ its outer and inner boundaries, respectively. For $1\leq j\leq 2N$, let $x_j^{\bullet}$ be a dual vertex of $\outb^{\bullet}$ that  is nearest to $x_{j}^{}$.
	
	We say that a sequence of discrete annular polygons $(\mathcal{A}^{\delta}; x_1^{\delta}, \ldots, x_{2N}^{\delta})$
	converges as $\delta \to 0$ to a polygon $(\mathcal{A}; x_1, \ldots, x_{2N})$ in the \textit{Carath\'{e}odory sense}
	if there exist covering maps $\varphi_{\delta}$ from $\unitD$ onto $\mathcal{A}^{\delta}$,
	and a covering map $\varphi$ from $\unitD$ onto $\mathcal{A}$,
	such that $\varphi_{\delta} \to \varphi$ locally uniformly on $\unitD$,
	and $\varphi_{\delta}^{-1}(x_j^{\delta}) \to \varphi^{-1}(x_j)$ for all $1\le j\le 2N$. 
	The convergence in the \textit{close-Carath\'{e}odory sense} is defined in the same way as for polygons where we only need to replace the simply connected domain $\Omega$ by the annular domain $\LA$. 
	
Later, we will consider Ising interfaces in annulus with boundary conditions~\eqref{eqn::bounda_condi} in Section~\ref{sec::Ising_annulus}. 

\subsection{Convergence of multiple Ising interfaces in polygons} 
\label{sec::Ising_polygon}

In this section, we recall conclusions about the convergence of Ising interfaces in polygons. 
For $\delta>0$, we will consider critical Ising model on a discrete dual  polygon $(\Omega^{\delta,\bullet};x_{1}^{\delta,\bullet},\ldots,x_{2N}^{\delta,\bullet})$ with alternating boundary condition: 
\begin{align} \label{eqn::bound_condi_polygon}
	\oplus\text{ on }\cup_{j=1}^N(x_{2j-1}^{\delta,\bullet}x_{2j}^{\delta,\bullet}),\quad \text{and}\quad \ominus\text{ on }\cup_{j=1}^N(x_{2j}^{\delta,\bullet}x_{2j+1}^{\delta,\bullet}).
\end{align} 
	For each $j\in \{1,2,\ldots, 2N\}$, we define the \textit{Ising interface} $\eta_j^{\delta}$ starting from $x_{j}^{\delta}$ as follows. If $j$ is odd, $\eta_j^{\delta}$ starts from $x_{j}^{\delta}$, traverses on the edges of $\Omega^{\delta}$, and turns at every dual-vertex in such a way that it always has dual-vertices with spin $\oplus$ on its left and spin $\ominus$ on its right. If there is an indetermination when arriving at a vertex, it turns left. If $j$ is even, the interface $\eta_j^{\delta}$ is defined similarly with the left/right switched.  When the discrete polygons $(\Omega^{\delta};x_1^{\delta},\ldots,x_{2N}^{\delta})$ converge in the close-Carath\'{e}odory sense, the tightness of the collection of interfaces $\{(\eta_1^{\delta},\ldots,\eta_{2N}^{\delta})\}_{\delta>0}$ under topology~\eqref{eq::curve_metric_2} is standard nowadays, see e.g.,~\cite[Theorem~4.1]{KarrilaMultipleSLELocalGlobal}; the key ingredient is RSW estimates, see~\cite{KemppainenSmirnovRandomCurves}.

In this section, we will describe the limiting distribution of the interfaces $(\eta_1^{\delta}, \ldots, \eta_{2N}^{\delta})$ under boundary condition~\eqref{eqn::bound_condi_polygon}. 
To this end, we first define the so-called total partition functions for the Ising model in polygons. For $N\geq 1$, we let $\Pi_N$ denote the set of all \textit{pair partitions} $\varpi=\{\{a_1,b_1\},\ldots,\{a_N,b_N\}\}$ of the set $\{1,2,\ldots,2N\}$, that is, partitions of this set into $N$ disjoint two-element subsets $\{a_j,b_j\}\subseteq \{1,2,\ldots,2N\}$, with the convention that 
\begin{align*}
	\text{$a_1<a_2<\cdots< a_N$ and $a_j<b_j$ for $j\in \{1,2,\ldots,N\}$.}
\end{align*}
We also denote by $\mathrm{sgn}(\varpi)$ the sign of the partition $\varpi$ defined as the sign of 
\begin{align*}
	\text{the product }	\prod (a-c)(a-d)(b-c)(b-d)\text{  over pairs of distinct elements $\{a,b\}, \{c,d\}\in \varpi$.}
\end{align*}
Define 
\[\totalH(\mathbb{H};\cdot)\colon \{(x_1,\ldots,x_{2N})\in\R^{2N}: x_1<\cdots<x_{2N}\}\to\R\]
by
\begin{equation}\label{eqn::Ising_totalpartitionfunction}
	\totalH(\mathbb{H};x_1,\ldots,x_{2N}):=\mathrm{Pf}\left(\frac{1}{x_k-x_j}\right)_{j,k=1}^{2N}=\sum_{\varpi \in \Pi_{N}}\mathrm{sgn}(\varpi)\sum_{\{a,b\}\in \varpi}\frac{1}{x_b-x_a}.
\end{equation}
This gives the definition of the total partition function for the polygon $(\HH; x_1, \ldots, x_{2N})$. Note that the total partition function $\totalH$ satisfies the M\"obius covariance rule~\eqref{eqn::COV} with $\kappa=3$ and $b=1/2$, see e.g.,~\cite[Proposition~4.6]{KytolaPeltolaPurePartitionFunctions}. We may extend the definition of $\totalH$ to general polygon $(\Omega; x_1, \ldots, x_{2N})$ whose marked points $x_1, \ldots, x_{2N}$ lie on sufficiently regular boundary segments ($C^{1+\eps}$ for some $\eps>0$) as in~\eqref{eqn::PartF_def_polygon} with $\kappa=3$ and $b=1/2$: 
\begin{equation}\label{eqn::Ising_total_polygon}
\PartF(\Omega; x_1, \ldots, x_{2N})=\prod_{j=1}^{2N}|\varphi'(x_j)|^{1/2}\times \PartF(\HH; \varphi(x_1), \ldots, \varphi(x_{2N})), 
\end{equation}
where $\varphi$ is any conformal map from $\Omega$ onto $\HH$ with $\varphi(x_1)<\cdots<\varphi(x_{2N})$. 

Let $(\Omega;x_1,\ldots,x_{2N})$ be a polygon and write $\boldsymbol{x}=(x_1,\ldots,x_{2N})$. We denote by $\PP_{\kappa=3\cond\alpha}^{(\Omega;\bs{x})}$ the law of chordal $N$-$\SLE_{\kappa}$ in polygon $(\Omega;\boldsymbol{x})$ associated to link pattern $\alpha\in\LP_N$ with $\kappa=3$. 
We then define a probability measure $\mathbb{P}_{\Ising}^{(\Omega;\bs{x})}$ on $2N$-tuples of curves $\bs{\eta}=(\eta_1,\ldots,\eta_{2N})$ in polygon $(\Omega; x_1,\ldots,x_{2N})$ by
\begin{equation}\label{eqn::def_ptotal}
	\PP_{\Ising}^{(\Omega;\boldsymbol{x})}:=\sum_{\alpha\in \LP_N}\frac{\mathcal{Z}_{\alpha}(\Omega;\bs{x})}{\totalH(\Omega;\bs{x})}\times 	\mathbb{P}_{\kappa=3\cond\alpha}^{(\Omega;\bs{x})}, 
\end{equation}
where $\PartF_{\alpha}$ is pure partition function~\eqref{eqn::PartF_def_polygon} with $\kappa=3, b=1/2$ and $\PartF$ is defined in~\eqref{eqn::Ising_totalpartitionfunction} and~\eqref{eqn::Ising_total_polygon}. 

\begin{proposition}\label{prop::Ising_polygon}
Fix a polygon $(\Omega;x_1,\ldots,x_{2N})$ and write $\boldsymbol{x}=(x_1,\ldots,x_{2N})$.  Suppose that a sequence of discrete polygons $(\Omega^{\delta};x_1^{\delta},\ldots,x_{2N}^{\delta})$ converges to $(\Omega;x_1,\ldots,x_{2N})$ in the close-Carath\'{e}odory sense. Consider the Ising model on $\Omega^{\delta,\bullet}$ with boundary condition~\eqref{eqn::bound_condi_polygon}. 
 For $j\in \{1,2,\ldots,2N\}$, let $\eta_j^{\delta}$ be the interface starting from $x_j^{\delta}$. Then  the law of the collection of interfaces $(\eta_1^{\delta},\ldots,\eta_{2N}^{\delta})$ converges weakly to $\PP_{\Ising}^{(\Omega;\bs{x})}$ under the topology induced by~\eqref{eq::curve_metric_2}.
\end{proposition}
\begin{proof}
	See \cite{IzyurovIsingMultiplyConnectedDomains, BeffaraPeltolaWuUniqueness, PeltolaWuCrossingProbaIsing, KarrilaNewProba}.
\end{proof}

\subsection{Convergence of multiple Ising interfaces in annulus} 
\label{sec::Ising_annulus}
In this section, we consider the convergence of Ising interfaces in annulus. For $\delta>0$, we will focus on critical Ising model on a discrete dual annular polygon $(\mathcal{A}^{\delta,\bullet};x_{1}^{\delta,\bullet},\ldots,x_{2N}^{\delta,\bullet})$ with boundary conditions~\eqref{eqn::bounda_condi}: 
	\begin{align*} 
		\oplus\text{ on }\cup_{j=1}^N(x_{2j-1}^{\delta,\bullet} x_{2j}^{\delta,\bullet}),\quad \ominus\text{ on }\cup_{j=1}^N(x_{2j}^{\delta,\bullet} x_{2j+1}^{\delta,\bullet}),\quad \text{and}\quad \text{free on }\intb^{\delta,\bullet}.
	\end{align*}
	We then define the Ising interface $\eta_j^{\delta}$ starting from $x_j^{\delta}$ as before.
The goal of this section is the following convergence of the collection of interfaces $(\eta_1^{\delta},\ldots,\eta_{2N}^{\delta})$.
		\begin{proposition} \label{prop::joint_Ising_annulus}
			Let $(\mathcal{A};x_1,\ldots,x_{2N})$ be an annular polygon and write $\boldsymbol{x}=(x_1,\ldots,x_{2N})$. Suppose that a sequence $(\mathcal{A}^{\delta};x_1^{\delta},\ldots,x_{2N}^{\delta})$ of discrete annular polygons converges to $(\mathcal{A};x_1,\ldots, x_{2N})$ in the close-Carath\'{e}odory sense. Consider the critical Ising model on $\mathcal{A}^{\delta,\bullet}$ with boundary conditions~\eqref{eqn::bounda_condi}. For $j\in \{1,2\ldots, 2N\}$, let $\eta_j^\delta$ be the interface starting from $x_j^\delta$. Then the law of the collection of interfaces $(\eta_1^{\delta},\ldots,\eta_{2N}^{\delta})$ converges weakly under the topology induced by~\eqref{eq::curve_metric_2}, and we denote by $\PP_{\Ising}^{(\mathcal{A};\boldsymbol{x})}$ the limiting distribution.
		\end{proposition}

In this setup, the Ising interfaces always have a convergent subsequence in the curve space $(X^{2N},\metric)$.
\begin{lemma}\label{lem::Ising_tightness}
	Assume the same setup as in 
	Proposition~\ref{prop::joint_Ising_annulus}.
	The family of laws of $\{(\eta_1^{\delta},\ldots,\eta_{2N}^{\delta})\}_{\delta>0}$ is 
	precompact in the space $(X^{2N},\metric)$ of curves under the metric~\eqref{eq::curve_metric_2}.
	Furthermore, given any subsequential limit $(\eta_1,\ldots,\eta_{2N})$, for each $j\in \{1,2,\ldots,2N\}$,  the curve $\eta_j$ satisfies one of the following two properties:
	\begin{itemize}
		\item either $\eta_j$ connecting $x_j$ to another marked points $x_k$, then $\eta_j$ is the time-reversal of $\eta_k$ and $\eta_j\cap \left(\partial\mathcal{A}\cup\left(\cup_{l\neq j,k}\eta_l\right)\right)=\emptyset$;
		\item or $\eta_j$ connecting $x_j$ to $\intb$ and $\eta_j\cap \left(\outb\cup\left(\cup_{k\neq j}\eta_k\right)\right)=\emptyset$. 
	\end{itemize}
\end{lemma}

\begin{proof}
	The proof of this lemma is standard nowadays, see, e.g.,~\cite[Proof of Corollary~1.2]{IzyurovIsingMultiplyConnectedDomains}. The key ingredient is RSW estimates, see~\cite{KemppainenSmirnovRandomCurves}. 
\end{proof}

The goal of this section is to prove Proposition~\ref{sec::Ising_annulus}. 
In fact, this convergence of the joint distributions of $(\eta_1^{\delta}, \ldots, \eta_{2N}^{\delta})$ is a combination of previous conclusions about the convergence of a single Ising interface in various setups. We recall them in Sections~\ref{sec::local_one_free} and~\ref{sec::local_annulus}.  

\subsubsection{Convergence of Ising interfaces in polygons with one free segment} 
\label{sec::local_one_free}
For $\delta>0$, we consider critical Ising model on a discrete dual polygon $(\Omega^{\delta,\bullet};x_1^{\delta,\bullet},\ldots,x_{n+2}^{\delta,\bullet})$ with the following boundary conditions ($x_{0}^{\delta,\bullet}:=x_{n+2}^{\delta,\bullet}$):
	\begin{align} \label{eqn::bounda_condi_one_free}
	(-1)^j \text{ on }(x_{j-1}^{\delta,\bullet}x_{j}^{\delta,\bullet}),\quad \text{for }j\in \{1,\ldots,n+1\},\quad\enspace \text{and}\quad \text{free on }(x_{n+1}^{\delta,\bullet}x_{n+2}^{\delta,\bullet}). 
\end{align}
For each $1\leq j\leq n$, we define Ising interface $\eta_j^{\delta}$ starting from $x_j^{\delta}$. For each $r>0$, we define $T_r^{\delta}(j)$ to be the first time that when $\eta_j^{\delta}$ gets within Euclidean distance $r$ from the other marked points $\{x_1^{\delta},\ldots,x_{j-1}^{\delta},x_{j+1}^{\delta},\ldots,x_{n}^{\delta}\}$ or the free boundary segment $(x_{n+1}^{\delta}x_{n+2}^{\delta})$. 

In order to describe the scaling limits of Ising interfaces under boundary conditions~\eqref{eqn::bounda_condi_one_free}, we have to define the corresponding partition function $\mathcal{R}$. We define
\begin{equation*}
	\mathcal{R}(\mathbb{H};\cdot)\colon \{(x_1,\ldots,x_{n+2})\in\R^{n+2}: x_1<\cdots<x_{n+2}\}\to\R
\end{equation*}
as follows. 	 
	If $n=2m$ with $m\geq 1$, we define 
	\begin{align}\label{eqn::Ising_polygon_onefree_even}
		\mathcal{R}(\mathbb{H};x_1,\ldots,x_{2m+2})=
		&\left(x_{2m+2}-x_{2m+1}\right)^{-1/8}\times \prod_{k=1}^{2m}\frac{1}{\sqrt{(x_{2m+2}-x_k)(x_{2m+1}-x_k)}}\notag\\
		&\times \sum_{\varpi\in\Pi_m}\mathrm{sgn}(\varpi)\prod_{\{a,b\}\in\varpi}\frac{(x_{2m+1}-x_a)(x_{2m+2}-x_b)+(x_{2m+1}-x_b)(x_{2m+2}-x_a)}{x_b-x_a},
	\end{align}
	where $\Pi_n$ and $\mathrm{sgn}(\varpi)$ are defined before~\eqref{eqn::Ising_totalpartitionfunction}. If $n=2m+1$ with $m\geq 0$, we define
	\begin{align}\label{eqn::Ising_polygon_onefree_odd}
		\mathcal{R}(\mathbb{H};x_1,\ldots,x_{2m+3})=&(x_{2m+3}-x_{2m+2})^{3/8}\times \prod_{k=1}^{2m+1}\frac{1}{\sqrt{(x_{2m+2}-x_k)(x_{2m+3}-x_k)}}\notag\\
		&\times \sum_{\varpi\in\Pi_{m+1}}\mathrm{sgn}(\varpi)\prod_{\{a,b\}\in \varpi\atop b\neq 2m+2} \frac{(x_{2m+2}-x_a)(x_{2m+3}-x_b)+(x_{2m+2}-x_b)(x_{2m+3}-x_a)}{x_b-x_a}.
	\end{align}
We may extend the definition of $\mathcal{R}$ to general polygon $(\Omega; x_1, \ldots, x_{n+2})$ whose marked points $x_1, \ldots, x_{n+2}$ lie on sufficiently regular boundary segments ($C^{1+\eps}$ for some $\eps>0$) as
	\begin{equation}\label{eqn::Ising_polygon_onefree_general}
	\mathcal{R}(\Omega;x_1,\ldots,x_{n+2}):=\prod_{j=1}^{n} |\varphi(x_j)|^{1/2}\times |\varphi(x_{n+1})|^{1/16}\times |\varphi(x_{n+2})|^{1/16}\times \mathcal{R}(\mathbb{H};\varphi(x_1),\ldots,\varphi(x_{n+2})),
	\end{equation}
	where $\varphi$ is any conformal map from $\Omega$ onto $\mathbb{H}$ so that $-\infty<\varphi(x_1)<\cdots<\varphi(x_{n+2})<\infty$. 
	In particular, for a polygon $(\mathbb{D};\exp(2\ii\theta^1),\ldots,\exp(2\ii\theta^{n+2}))$ with $\theta^1<\cdots<\theta^{n+2}<\theta^{1}+\pi$, we write
	\begin{equation}\label{eqn::Ising_polygon_onefree_unitD}
		\mathcal{R}(\theta^1,\ldots,\theta^{n+2})=\mathcal{R}(\mathbb{D};\exp(2\ii\theta^1),\ldots,\exp(2\ii\theta^{n+2})). 
\end{equation}
Then we have the following local convergence of a single interface:
\begin{lemma} \label{lem::local_polygon_one_free}
	Fix a polygon $(\Omega;x_1,\ldots,x_{n+2})$. Let $\varphi$ be a conformal map from $\unitD$ onto $\Omega$ and write $\varphi^{-1}(x_j)=\exp(2\ii\theta^j)$ with $\theta^1<\cdots<\theta^{n+2}<\theta^1+\pi$. Suppose that a sequence of discrete polygons $(\Omega^{\delta};x_1^{\delta},\ldots,x_{n+2}^{\delta})$ converges to $(\Omega;x_1,\ldots,x_{n+2})$ in the close-Carath\'{e}odory sense. Consider the Ising model on $\Omega^{\delta,\bullet}$ with boundary conditions~\eqref{eqn::bounda_condi_one_free}. Fix $j\in \{1,2,\ldots,n\}$, let $\eta_j^{\delta}$ be the interface starting from $x_j^{\delta}$. Then for each $r>0$, the curve $(\eta_j^{\delta}(t),0\leq t\leq T_r^{\delta}(j))$ converges weakly under the topology induced by~\eqref{eq::curve_metric_1} to the image under $\varphi$ of the Radial Loewner chain $\mathring{\eta}_j$ with the following driving function with the $\frac{2}{3}$-usual parameterization, up to the first time that the $r$-neighborhood of $\{x_1,\ldots,x_{j-1},x_{j+1},\ldots,x_n\}\cup (x_{n+1}x_{n+2})$ is hit:
		\begin{equation}\label{eqn::local_polygon_one_free_SDE}
		\begin{cases}
			\ud \xi_t^j=\ud B_t+(\partial_j \log\mathcal{R})(V_t^1,\ldots,V_t^{j-1},\xi_t^j,V_t^{j+1},\ldots,V_t^{n+2})\ud t,\quad	\xi_0^j=\theta^j,\\
			\ud V_t^k=\frac{2}{3}\cot(V_t^k-\xi_t^j)\ud t,\quad  V_0^k=\theta^k,\enspace \text{for }k\in \{1,2,\ldots,n+2\}\setminus \{j\},
		\end{cases}
	\end{equation}
where $B$ is a standard Brownian motion and $\LR$ is defined in~\eqref{eqn::Ising_polygon_onefree_even}-\eqref{eqn::Ising_polygon_onefree_unitD}. 
\end{lemma}
\begin{proof}
	This is a special case of~\cite[Theorem~3.1]{IzyurovObservableFree}.
\end{proof}

\subsubsection{Convergence of Ising interfaces in annulus} \label{sec::local_annulus}
Now, for $\delta>0$, we consider critical Ising model on a dual annular polygon $(\mathcal{A}^{\delta,\bullet};x_1^{\delta,\bullet},\ldots,x_{2N}^{\delta,\bullet})$ with boundary conditions~\eqref{eqn::bounda_condi}. Let $j\in \{1,2,\ldots,2N\}$. For each $r>0$, we define $T_r^{\delta}(j)$ to be the first time that when $\eta_j^{\delta}$ gets within Euclidean distance $r$ from the other marked points $\{x_1^{\delta},\ldots,x_{j-1}^{\delta},x_{j+1}^{\delta},\ldots,x_{2N}^{\delta}\}$ or the free inner boundary $\intb^\delta$. 

In order to describe the scaling limits of Ising interfaces under boundary conditions~\eqref{eqn::bounda_condi}, we have to define the corresponding partition function $\mathcal{H}$.
 For $p>0$ and an annular polygon $(\mathbb{A}_p;\exp(2\ii\theta^1),\ldots,\exp(2\ii\theta^{2N}))$ with $\theta^1<\cdots<\theta^{2N}<\theta^1+\pi$,  we write $x_j=\exp(2\ii\theta^j)$ for $1\leq j\leq 2N$. We define 
	\begin{align} \label{eqn::def_H}
		\mathcal{H}(\mathbb{A}_p;\theta^1,\ldots,\theta^{2N}):=&\mathrm{Pf}\Big(\left(2/\sqrt{x_jx_k}\right)\cdot \mathrm{cs}\left(\log\left(x_j/x_k\right)\right)\Big)_{j,k=1}^{2N}\notag\\
		=&\sum_{\varpi \in \Pi_{N}}\mathrm{sgn}(\varpi)\sum_{\{a,b\}\in \varpi}\left(2/\sqrt{x_ax_b}\right)\cdot \mathrm{cs}\left(\log\left(x_a/x_b\right)\right),
	\end{align}
	where $\mathrm{cs}$ is the Jacobian elliptic function with half-periods $2K=2p$ and $2K'=2\pi\ii$ (see e.g.,~\cite[Section~16.1]{AbramowitzHandbook}), and where $\Pi_N$ and $\mathrm{sgn}(\varpi)$ are defined before~\eqref{eqn::Ising_totalpartitionfunction}. 
		Now consider a general annular polygon $(\mathcal{A};x_1,\ldots,x_{2N})$ with $\partial_{\mathrm{o}}=\partial \mathbb{D}$ and $x_j=\exp(2\ii\theta^j)$ for $j\in \{1,2,\ldots,2N\}$ with $\theta^1<\cdots<\theta^{2N}<\theta^1+\pi$. Suppose that $\mathcal{A}$ is conformally equivalent to some $\mathbb{A}_p$. We extend the definition of $\mathcal{H}$ via conformal image:
	\begin{equation} \label{eqn::H_general}
		\mathcal{H}(\mathcal{A};\theta^1,\ldots,\theta^{2N}):=\left(\prod_{j=1}^{2N}|\varphi'(x_j)|^{1/2}\right)\times \mathcal{H}(\mathbb{A}_p;\tilde{\theta}^1,\ldots,\tilde{\theta}^{2N}),
	\end{equation}
	where $\varphi$ is any conformal map from $\mathcal{A}$ onto $\mathbb{A}_p$ such that $\varphi(\partial\mathbb{D})=\partial\mathbb{D}$, and $\exp(2\ii\tilde{\theta}_j)=\varphi(x_j)$ where $\tilde{\theta}^1<\cdots<\tilde{\theta}^{2N}<\tilde{\theta}^1+\pi$.
	Then we have the following local convergence of a single interface:
	\begin{lemma} \label{lem::local_annulus}
		Assume the same setup as in Proposition~\ref{prop::joint_Ising_annulus}. Let $\varphi$ be a conformal map from some $\mathbb{A}_p$ onto $\mathcal{A}$ and write $\varphi^{-1}(x_j)=\exp(2\ii\theta^j)$ with $\theta^1<\cdots<\theta^{2N}<\theta^1+\pi$.  Fix $j\in \{1,2,\ldots,2N\}$. 
		 Then for each $r>0$, the curve $(\eta_j^{\delta}(t),0\leq t\leq T_r^{\delta}(j))$ converges weakly under the topology induced by~\eqref{eq::curve_metric_1} to the image under $\varphi$ of the Radial Loewner chain $\mathring{\eta}_j$ with the following driving function with the $\frac{2}{3}$-usual parameterization, up to the first time that the $r$-neighborhood of $\{x_1,\ldots,x_{j-1},x_{j+1},\ldots,x_{2N}\}\cup \intb$ is hit:
				\begin{equation} \label{eqn::loca_annulus_SDE}
			\begin{cases}
				\ud \xi_t^j=\ud B_t+(\partial_j \log\mathcal{H})(\Lambda_t^j; V_t^1,\ldots,V_t^{j-1},\xi_t^j,V_t^{j+1},\ldots,V_t^{2N})\ud t,\quad	\xi_0^j=\theta^j,\\
				\ud V_t^k=\frac{2}{3}\cot(V_t^k-\xi_t^j)\ud t,\quad  V_0^k=\theta^k,\enspace \text{for }k\in \{1,2,\ldots,2N\}\setminus \{j\},
			\end{cases}
		\end{equation}
	where $B$ is a standard Brownian motion, $\mathcal{H}$ is defined in~\eqref{eqn::def_H}-\eqref{eqn::H_general} and $\Lambda_t^j:=g_t^j\left(\mathbb{A}_p\setminus\mathring{\eta}_j\left([0,t]\right)\right)$, and where $g_t^j$ is the radial Loewner map associated with $\mathring{\eta}_j$.
	\end{lemma}
\begin{proof}
This is a special case of~\cite[Corollary~1.5]{CHI22}.
\end{proof}

\subsubsection{From local convergence to global convergence}
The convergence of interfaces in Sections~\ref{sec::local_one_free}-\ref{sec::local_annulus} is local in the sense that we only consider initial segments of the interface $\eta_j^{\delta}$ restricted to a neighborhood of $x_j^{\delta}$. However, in order to prove Proposition~\ref{prop::joint_Ising_annulus}, we need the convergence of the whole interface $\eta_j^{\delta}$.  We follow the framework of~\cite[Section~5]{KarrilaMultipleSLELocalGlobal}.

\begin{lemma} \label{lem::polygon_localtoglobal}
	Assume the same setup as in Lemma~\ref{lem::local_polygon_one_free}. Then the curve $\eta_j^{\delta}$ converges weakly under the topology induced by~\eqref{eq::curve_metric_1} to the image under $\varphi$ of a continuous random curve $\mathring{\eta}_j$ starting from $x_j$. Moreover, $\mathring{\eta}_j$ is a radial Loewner chain, when using $\frac{2}{3}$-usual parameterization, it has the driving function satisfying the SDEs~\eqref{eqn::local_polygon_one_free_SDE}, up to the first time $\partial\unitD\setminus\{x_j\}$ is hit. 
\end{lemma}

\begin{lemma}\label{lem::annulus_localtoglobal}
	Assume the same setup as in Lemma~\ref{lem::local_annulus}. Then the curve $\eta_j^{\delta}$ converges weakly under the topology induced by~\eqref{eq::curve_metric_1} to the image under $\varphi$ of a continuous curve $\mathring{\eta}_j$ starting from $x_j$. Moreover, $\mathring{\eta}_j$ is a radial Loewner chain, when using $\frac{2}{3}$-usual parameterization, it has the driving function satisfying the SDEs~\eqref{eqn::loca_annulus_SDE}, up to the first time $\partial\mathcal{A}\setminus \{x_j\}$ is hit.
\end{lemma}

\begin{proof}[Proof of Lemmas~\ref{lem::polygon_localtoglobal} and~\ref{lem::annulus_localtoglobal}]
The proof is essentially the same as the proof of~\cite[Theorem~5.8 \& Proposition~6.1]{KarrilaMultipleSLELocalGlobal}, with the following inputs:
\begin{itemize}
	\item Local convergence of $\eta_j^{\delta}$ under boundary conditions~\eqref{eqn::bounda_condi_one_free} and~\eqref{eqn::bounda_condi} in Lemmas~\ref{lem::local_polygon_one_free} and~\ref{lem::local_annulus}, respectively, which corresponds to~\cite[Assumption~5.1]{KarrilaMultipleSLELocalGlobal};
	\item RSW estimates and domain Markov property for Ising model, which correspond to~\cite[Assumption~5.6]{KarrilaMultipleSLELocalGlobal}.
\end{itemize}
\end{proof}
%

%
%
\subsubsection{Proof of Proposition~\ref{prop::joint_Ising_annulus}}
With tightness in Lemma~\ref{lem::Ising_tightness} and  convergence of a single interface in Lemmas~\ref{lem::polygon_localtoglobal} and~\ref{lem::annulus_localtoglobal}, we are ready to prove Proposition~\ref{prop::joint_Ising_annulus}.

\begin{proof}[Proof of Proposition~\ref{prop::joint_Ising_annulus}]
	By Lemma~\ref{lem::Ising_tightness} (also by coupling them into the same probability space), we may choose a subsequence $\delta_n\to 0$ such that $(\eta_1^{\delta_n},\ldots,\eta_{2N}^{\delta_n})$ converges to $(\eta_1,\ldots,\eta_{2N})$ in metric~\eqref{eq::curve_metric_2} as $n\to \infty$, almost surely. We denote by $\tau$ the hitting time of $\partial\mathcal{A}\setminus \{x_1\}$ by $\eta_1$, and denote by $\tau_n$ the hitting time of $\partial\mathcal{A}^{\delta_n}$ by $\eta_1^{\delta_n}$. 
	 There are two cases considering the location of $\eta_1(\tau)$ (resp., $\eta_1^{\delta_n}(\tau_n)$): 
	 \begin{itemize}
	 	\item $\eta_1(\tau)=x_k$ (resp., $\eta_1^{\delta_n}(\tau_n)=x_k^{\delta_n}$) for some even $k$ and we denote by $\mathcal{A}_L$ and $\mathcal{A}_R$ (resp., $\mathcal{A}_L^{\delta_n}$ and $\mathcal{A}_R^{\delta_n}$) the connected components of $\mathcal{A}\setminus \eta_1$ (resp., $\mathcal{A}^{\delta_n}\setminus \eta_1^{\delta_n}$) that are on the left and right hand sides of $\eta_1$ (resp., $\eta_1^{\delta_n}$), respectively;
	 	\item  $\eta_1(\tau)\in \intb$ (resp., $\eta_1^{\delta_n}(\tau_n)\in \intb^{\delta_n}$) and we denote by $x_L^{*}$ and $x_R^{*}$ (resp., $x_L^{(*,\delta_n)}$ and $x_R^{(*,\delta_n)}$) the primal ends of $\partial\left(\mathcal{A}\setminus \eta_1\right)$ (resp., $\partial\left(\mathcal{A}^{\delta_n}\setminus \eta_1^{\delta_n}\right)$) correspond to $\eta_1(\tau)$ (resp., $\eta_1^{\delta_n}(\tau_n)$) lying on the left and right hand sides of $\eta_1$ (resp., $\eta_1^{\delta_n}$), respectively.
	 \end{itemize}
	 Thanks to Lemmas~\ref{lem::local_annulus} and~\ref{lem::annulus_localtoglobal}, the law of $\eta_1$ is independent of the subsequence $\{\delta_n\}_{n\geq 1}$. It remains to show that 
	 \begin{equation}\label{eqn::joint_annulus_aux1}
	 \text{given $\eta_1$, the conditional law of $(\eta_2,\ldots,\eta_{2N})$ is also independent of the subsequence $\{\delta_n\}_{n\geq 1}$. }
	 \end{equation} 
	We proceed by induction on $N\geq 1$.
	
	First, we consider the base case $N=1$. On the event $\{\eta_1(\tau)=x_2\}$,  the curve $\eta_2$ is the time-reversal of $\eta_1$, so that the conditional law of $\eta_2$ given $\eta_1$ is independent of $\{\delta_n\}_{n\geq 1}$, as desired. 
		On the event $\{\eta_1(\tau)\in \intb\}$, without loss of generality, we may assume that $\eta_1^{\delta_n}(\tau_n)\in \intb^{\delta_n}$ for all $n\geq 1$. 
	Then one can proceed as in~\cite[Proof of Lemma~5.6]{Izy22} to show that the discrete polygons $(\mathcal{A}^{\delta_n}\setminus \eta_1^{\delta_n};x_2^{\delta_n},x_L^{(*,\delta_n)},x_R^{(*,\delta_n)})$ converge almost surely to $(\mathcal{A}\setminus\eta_1;x_2,x_{L}^*,x_R^*)$ in the  close-Carath\'{e}odory sense. Thus, almost surely, there exist conformal maps $\varphi_n$ (resp., $\varphi$) from $\mathbb{D}$ onto $\mathcal{A}^{\delta_n}\setminus \eta_1^{\delta_n}$ (resp., $\mathcal{A}\setminus \eta_1$) such that, as $n\to \infty$, the maps $\varphi_n$ converge uniformly on compact subsets of $\mathbb{D}$ to $\varphi$ and we have $(\varphi_n)^{-1}(x_2^{\delta_n})\to (\varphi^{-1})(x_2)$, $(\varphi_n)^{-1}(x_L^{(*,\delta_n)})\to (\varphi^{-1})(x_L^*)$ and $(\varphi_n)^{-1}(x_R^{(*,\delta_n)})\to (\varphi^{-1})(x_R^*)$. 
	\begin{itemize}
		\item On the one hand, thanks to the domain Markov property of Ising model, Lemmas~\ref{lem::local_polygon_one_free} and~\ref{lem::polygon_localtoglobal}, the law of $(\varphi_n)^{-1}({\eta_2^{\delta_n}})$ converges weakly under the topology induced by~\eqref{eq::curve_metric_2} to a limit that is independent of the subsequence $\{\delta_n\}_{n\geq 1}$.
		\item On the other hand, thanks to~\cite[Theorem~4.2]{KarrilaConformalImage} and the hypothesis that $(\eta_1^{\delta_n},\eta_2^{\delta_n})$ converges almost surely to $(\eta_1,\eta_2)$ in the metric~\eqref{eq::curve_metric_2}, one can proceed as in~\cite[Proof of Lemma~4.7]{BeffaraPeltolaWuUniqueness} to show that $(\varphi_n)^{-1}(\eta_2^{\delta_n})$ converges weakly to $\varphi^{-1}(\eta_2)$ under the topology induced by~\eqref{eq::curve_metric_2}.
	\end{itemize}
	Thus, on the event $\{\eta_1(\tau)\in \intb\}$, the conditional law of $\eta_2$ given $\eta_1$ is independent of the subsequence $\{\delta_n\}_{n\geq 1}$. To sum up,~\eqref{eqn::joint_annulus_aux1} holds for the base case $N=1$.
	
	Second, let us consider the case $N\geq 2$. We make the induction hypothesis that~\eqref{eqn::joint_annulus_aux1} holds for all $N'\leq N-1$. We now consider the case $N$. 
\begin{itemize}
\item On the event $\{\eta_1(\tau)=x_k\}$, without loss of generality, we may assume that $\eta_1^{\delta_n}$ hits $x_k^{\delta_n}$ for all $n\geq 1$. 
One can proceed as in~\cite[Proof of Lemma~5.6]{Izy22} to show that $(\mathcal{A}_L^{\delta_n};x_{k+1}^{\delta_n},\ldots,x_{2N}^{\delta_n})$ (resp., $(\mathcal{A}_R^{\delta_n};x_{2}^{\delta_n},\ldots,x_{k-1}^{\delta_n})$) converges almost surely to $(\mathcal{A}_L;x_{k+1},\ldots,x_{2N})$ (resp., $(\mathcal{A}_R;x_2,\ldots,x_{k-1})$) in the close-Carath\'{e}odory sense. Then, thanks to the domain Markov property of Ising model, the induction hypothesis,  and Proposition~\ref{prop::Ising_polygon}, one can proceed as in the case of $N=1$ to show that  given $\eta_1$, on the event $\{\eta_1(\tau)=x_k\}$, the conditional law of $(\eta_2,\ldots,\eta_{k-1},\eta_{k+1},\ldots,\eta_{2N})$ is the independent product of $\mathbb{P}_{\Ising}^{(\mathcal{A}_L;x_{k+1},\ldots,x_{2N})}$ and $\mathbb{P}_{\Ising}^{(\mathcal{A}_R;x_2,\ldots,x_{k-1})}$ and in particular, is independent of the subsequence $\{\delta_n\}_{n\geq 1}$.
\item On the event $\{\eta_1(\tau)\in \intb\}$, one can proceed as in~\cite[Proof of Lemma~5.6]{Izy22} to show that the discrete polygons $(\mathcal{A}^{\delta_n}\setminus \eta_1^{\delta_n};x_2^{\delta_n}, \ldots, x_{2N}^{\delta_n}, x_L^{(*,\delta_n)},x_R^{(*,\delta_n)})$ converge almost surely to $(\mathcal{A}\setminus\eta_1;x_2,\ldots, x_{2N}, x_{L}^*,x_R^*)$ in the  close-Carath\'{e}odory sense. Then, thanks to the domain Markov property of Ising model and Lemmas~\ref{lem::local_polygon_one_free} and~\ref{lem::polygon_localtoglobal}, one can proceed as in the case of $N=1$ to show that given $\eta_1$, on the event $\{\eta_1(\tau)\in\intb\}$, the conditional law of $(\eta_2, \ldots, \eta_{2N})$ is also independent of $\{\delta_n\}_{n\ge 1}$.  
\end{itemize}	
To sum up,~\eqref{eqn::joint_annulus_aux1} also holds for the case $N$, as we set out to prove. 
\end{proof}

\subsection{Proof of Theorem~\ref{thm::Ising_annulus_radialSLE}} 
\label{sec::cvg_Ising_radial}
	In this section, we combine Propositions~\ref{prop::Ising_polygon} and~\ref{prop::joint_Ising_annulus} with tools from Section~\ref{sec::multipleSLE} to show Theorem~\ref{thm::Ising_annulus_radialSLE}. We fix parameters
	\begin{equation*}
		\kappa=3,\quad a=\frac{2}{3},\quad b=\frac{1}{2}.
	\end{equation*}
	
Fix $\theta^1<\cdots<\theta^{2N}<\theta^1+\pi$ and write $\bs{\theta}=(\theta^1, \ldots, \theta^{2N})$. Denote $x_j=\exp(2\ii \theta^j)$ for $1\leq j\leq 2N$ and write $\boldsymbol{x}=(x_1,\ldots,x_{2N})$.
\begin{itemize}
	\item Recall that $\PP_{\Ising}^{(\mathbb{D};\boldsymbol{x})}$ in Proposition~\ref{prop::Ising_polygon} denotes the limiting distribution of Ising interfaces in polygon $(\unitD; x_1, \ldots, x_{2N})$. 
	\item Recall that $\PP_{\Ising}^{(\A_p;\boldsymbol{x})}$ in Proposition~\ref{prop::joint_Ising_annulus} denotes the limiting distribution of Ising interfaces in annular polygon $(\A_p; x_1, \ldots, x_{2N})$. 
\end{itemize}
We denote 
\[E_p=\{\eta_1, \ldots, \eta_{2N}\text{ all hit }e^{-p}\unitD\}. \]

\begin{lemma} \label{lem::compare_arm_events}
There exist two constants $c_1,c_2>0$ that only depend on $N$ so that 
\begin{equation} \label{eqn::IsingAsy_aux2}
	c_1\leq \frac{\mathbb{P}_{\Ising}^{(\A_p;\bs	{x})}[E_p]}{\mathbb{P}_{\Ising}^{(\unitD;\bs{x})}[E_p]}\leq c_2. 
\end{equation}
\end{lemma}
We postpone the proof of Lemma~\ref{lem::compare_arm_events} to the next section. Assuming Lemma~\ref{lem::compare_arm_events}, we are ready to prove Theorem~\ref{thm::Ising_annulus_radialSLE}. For $\delta>0$ and $r\in (0,1]$, let $\Gamma_r^{\delta} $  be a closed simple loop consisting of edges in $\delta\mathbb{Z}^2$ so that $\Gamma_r^{\delta}$ tends to $r\partial\unitD$ as $\delta$ tends to $0$ in the Hausdorff metric.  Let $\Lambda_r^{\delta}$ be the discrete annular domain with outer boundary $\Gamma_1^{\delta}$ and inner boundary  $\Gamma_r^{\delta}$. From now on, we fix a choice of $\Gamma_{r}^{\delta}$ and hence $\Lambda_{r}^{\delta}$ for $\delta>0$ and $r\in (0,1)$.

\begin{proof}[Proof of~\eqref{eqn::IsingAsy}]
		According to Theorem~\ref{thm::NChordalSLEGreenFunction} and~\eqref{eqn::def_ptotal}, there exists a constant $\tilde{C}\in (1,\infty)$ that may depend on $\boldsymbol{x}$ so that 
	\begin{equation} \label{eqn::IsingAsy_aux1}
		\tilde{C}^{-1}\exp\left(-\frac{16N^2-1}{24}p\right)\le \PP_{\Ising}^{(\unitD;\bs{x})}\left[\eta_1, \ldots, \eta_{2N}\text{ all hit } e^{-p}\unitD\right]\le \tilde{C}\exp\left(-\frac{16N^2-1}{24}p\right). 
	\end{equation}
Then~\eqref{eqn::IsingAsy} follows readily from Lemma~\ref{lem::compare_arm_events} and~\eqref{eqn::IsingAsy_aux1}.
\end{proof}

\begin{proof}[Proof of~\eqref{eqn::IsingAnnulusCvg}]
		We will compare the following three measures: 
	\begin{itemize}
		\item Recall that $\PP_{\mathrm{Ising}\cond p}^{(\A_p;\bs{x})}$ denotes the law of the limit $(\eta_1, \ldots, \eta_{2N})$ in Proposition~\ref{prop::joint_Ising_annulus} conditional on $E_p$. 
		\item Recall that $\PP_{\kappa=3\cond *}^{(\bs{\theta})}$ denotes the law of $2N$-sided radial $\SLE_{3}$ in polygon $(\unitD; \exp(2\ii\theta^1), \ldots, \exp(2\ii\theta^{2N}))$. 
		\item Denote by $\PP_{\Ising\cond p}^{(\unitD;\bs{x})}$ the law of $(\eta_1,\ldots,\eta_{2N})$ under $\mathbb{P}^{(\unitD;\bs{x})}_{\Ising}$ in Proposition~\ref{prop::Ising_polygon} conditional on $E_p$.  
	\end{itemize}

	We parameterize $\bs{\eta}=(\eta_1, \ldots, \eta_{2N})$ by 
$\frac{2}{3}$-common parameterization and denote by $(\LF_t, t\ge 0)$ the filtration generated by $(\boldsymbol{\eta}(t), t\ge 0)$. Fix any $s>0$, when all three measures are restricted to $\LF_s$, the total variation distances go to zero as $p\to\infty$: 
\begin{equation}\label{eqn::TV_Total_Radial}
	\lim_{p\to \infty}\dist_{TV}\left(\PP_{\Ising\cond p}^{(\unitD;\bs{x})}[\cdot \cond_{\LF_s}],\PP_{\kappa=3\cond *}^{(\bs{\theta})}[\cdot\cond_{\LF_s}]\right)=0;
\end{equation}
\begin{equation}\label{eqn::TV_Ising_Total}
	\lim_{p\to \infty}\dist_{TV}\left(\PP_{\mathrm{Ising}\cond p}^{(\A_p;\bs{x})}[\cdot \cond_{\LF_s}],\PP_{\Ising\cond p}^{(\unitD;\bs{x})}[\cdot \cond_{\LF_s}]\right)=0.
\end{equation}
The conclusion follows from these two convergences. The convergence in~\eqref{eqn::TV_Total_Radial} follows from Corollary~\ref{cor::chordaltoradial} and~\eqref{eqn::def_ptotal}. To show the convergence in~\eqref{eqn::TV_Ising_Total}, it suffices to show that 
\begin{equation} \label{eqn::Ising_annulus_radialSLE_aux21}
	\lim_{p\to \infty}	\sup_{E\in \mathcal{F}_s}\left| \frac{\PP_{\mathrm{Ising}}^{(\A_p;\bs{x})}[E\cap E_p]}{\PP_{\mathrm{Ising}}^{(\A_p;\bs{x})}[E_p]}-\frac{\PP_{\Ising}^{(\unitD;\bs{x})}[E\cap E_p]}{\PP_{\Ising}^{(\unitD;\bs{x})}[E_p]}\right|=0.
\end{equation}

Now let us prove~\eqref{eqn::Ising_annulus_radialSLE_aux21}.   
\begin{itemize}
\item On the one hand, let $(\mathcal{A}^{\delta};x_{1}^{\delta},\ldots,x_{2N}^{\delta})$ be a sequence of discrete annular polygons so that $(x_{j}^{\delta}x_{j+1}^{\delta})$ converges to $(x_j x_{j+1})$ for each $j$ under the Hausdorff distance and that $\mathcal{A}^{\delta}=\Lambda_{e^{-p}}^{\delta}$. Let $\mathcal{A}^{\delta,\bullet}$ be the dual graph of $\mathcal{A}^{\delta}$ and define $x_{j}^{\delta,\bullet}$ to be a vertex in $(\delta\mathbb{Z}^2)^{\bullet}$ that is nearest to $x_j^{\delta}$ for $1\leq j\leq 2N$.
Denote by $\PP^{\delta}_{\LA}$ the Ising measure on $\mathcal{A}^{\delta,\bullet}$ with boundary conditions~\eqref{eqn::bounda_condi}. Then, according to Proposition~\ref{prop::joint_Ising_annulus},  the law of the collection of Ising interfaces $(\eta_1^{\delta},\ldots,\eta_{2N}^{\delta})$ under $\PP^{\delta}_{\LA}$ converges weakly to $\PP_{\mathrm{Ising}}^{(\A_p;\bs{x})}$ as $\delta\to0$ under the topology induced by~\eqref{eq::curve_metric_2}. 
\item On the other hand, denote by $\Omega^{\delta}$ the simply connected subgraph of $\delta\mathbb{Z}^2$ induced by the vertices lying on or enclosed by the outer boundary of $\mathcal{A}^{\delta}$. Let $\Omega^{\delta,\bullet}$ be the dual graph of $\Omega^{\delta}$ and define $x_{j}^{\delta,\bullet}$ to be a vertex in $(\delta\mathbb{Z}^2)^{\bullet}$ that is nearest to $x_j^{\delta}$ for $1\leq j\leq 2N$. Denote by ${\mathbb{P}}_{\Omega}^{\delta}$ the Ising measure on $\Omega^{\delta,\bullet}$ with boundary conditions~\eqref{eqn::bound_condi_polygon}.
According to Proposition~\ref{prop::Ising_polygon}, the law of the collection of Ising interfaces $(\eta_1^{\delta},\ldots,\eta_{2N}^{\delta})$ under ${\mathbb{P}}_{\Omega}^{\delta}$ converges weakly to $\PP_{\Ising}^{(\unitD;\bs{x})}$ as $\delta\to0$ under the topology induced by~\eqref{eq::curve_metric_2}. 
\end{itemize}

We denote
\[E_p^{\delta}=\{\eta_1^{\delta},\ldots,\eta_{2N}^{\delta}\text{ all hit }\Gamma_{e^{-p}}^{\delta}\}.\]
Thanks to Lemma~\ref{lem::Ising_coupling} below (proven in Section~\ref{sec::technical}), in the discrete, there exist constants $c_3, c_4>0$ that depend only on $\boldsymbol{x}$ so that for any $r\in (e^{-p},1)$,  small enough $\delta>0$ and  any event $E^{\delta}$ depends on status inside $\Lambda_{r}^{\delta}$, we have 
\begin{equation}\label{eqn::Ising_annulus_radialSLE_aux22}
	\left|\PP^{\delta}_{\LA}[E^{\delta}|E_p^{\delta}]-\mathbb{P}_{\Omega}^{\delta}[E^{\delta}|E_p^{\delta}]\right|\leq c_3(\frac{e^{-p}}{r})^{c_4}.
\end{equation}
According to  Koebe's $\frac{1}{4}$ theorem and the definition of $\frac{2}{3}$-common parameterization, we have 
\begin{equation}\label{eqn::Ising_annulus_radialSLE_aux23}
\dist\left(0,\eta_j[0,s]\right)\geq \frac{1}{4}e^{-\frac{4Ns}{3}},\quad \forall 1\leq j\leq 2N.
\end{equation}
Combining convergence of discrete interfaces with~\eqref{eqn::Ising_annulus_radialSLE_aux22}-\eqref{eqn::Ising_annulus_radialSLE_aux23}, we have that for $E\in\LF_s$,
\[\left| \frac{\PP_{\mathrm{Ising}}^{(\A_p;\bs{x})}[E\cap E_p]}{\PP_{\mathrm{Ising}}^{(\A_p;\bs{x})}[E_p]}-\frac{\PP_{\Ising}^{(\unitD;\bs{x})}[E\cap E_p]}{\PP_{\Ising}^{(\unitD;\bs{x})}[E_p]}\right|\leq c_3(4e^{-p+\frac{4Ns}{3}})^{c_4}.\]
This proves~\eqref{eqn::Ising_annulus_radialSLE_aux21} and completes the proof of~\eqref{eqn::IsingAnnulusCvg}.
\end{proof}

\begin{lemma}\label{lem::Ising_coupling}
		Let $\PP^{\delta}_{\LA}$, $\PP^{\delta}_{\Omega}$ and $E_p^{\delta}$ be as in the proof of~\eqref{eqn::IsingAnnulusCvg}. Then there exist two constants $c_3,c_4>0$ that depend only on $\boldsymbol{x}$ so that for any $r\in (e^{-p},1)$ and any  small enough $\delta>0$, there exist a coupling $\hat{\mathbb{P}}^{\delta}$ of two conditional measures $\PP^{\delta}_{\LA}[\cdot | E_p], \mathbb{P}_{\Omega}^{\delta}[\cdot |E_p]$ and an event $\mathcal{O}_r^{\delta}$ so that
		\begin{equation*}
			\hat{\mathbb{P}}^{\delta}[\mathcal{O}_r^{\delta}]\geq 1-c_3(\frac{e^{-p}}{r})^{c_4},
		\end{equation*}  
		and that if $\mathcal{O}_r^{\delta}$ happens, then status of vertices in $\Lambda_{r}^{\delta}$ are the same under both measures $\PP^{\delta}_{\LA}[\cdot |E_p]$ and $\mathbb{P}_{\Omega}^{\delta}[\cdot| E_p]$. In particular, for any event $E^{\delta}$ depending on the status of vertices inside $\Lambda_{r}^{\delta}$, we have~\eqref{eqn::Ising_annulus_radialSLE_aux22}.
	\end{lemma}

\begin{proof}[Proof of~\eqref{eqn::IsingAnnulusDrivingCvg}]
	This is immediate from~\eqref{eqn::IsingAnnulusCvg} because $\nu_{\mathrm{Ising}\cond p}^{(\bs{\theta})}$ is the induced measure on the driving function $(\bs{\theta}_t, t\ge 0)$ under $\PP_{\mathrm{Ising}\cond p}^{(\unitD;\bs{x})}$ and $\nu_{\kappa=3|*}^{(\bs{\theta})}$ is the induced measure on the driving function under $\PP_{\kappa=3\cond *}^{(\bs{\theta})}$. 
\end{proof}

\subsection{Proof of Lemmas~\ref{lem::compare_arm_events} and~\ref{lem::Ising_coupling}} 
\label{sec::technical}
The goal of this section is to prove Lemmas~\ref{lem::compare_arm_events} and~\ref{lem::Ising_coupling}. We first collect some properties of critical Ising model that we will use: a strong RSW estimates Lemma~\ref{lem::RSW}, spatial mixing property Lemma~\ref{lem::spatial} and well-separatedness of arm events Lemma~\ref{lem::Ising_well_separa}.

A discrete topological rectangle $(\Omega;a,b,c,d)$ is a discrete polygon with $4$ marked points. We denote by $\mathrm{d}_{\Omega}\left(\left(ab\right),\left(cd\right)\right)$ the discrete extremal distance between $(ab)$ and $(cd)$ in $\Omega$; see~\cite[Section~6]{ChelkakRobust}. The discrete extremal distance is uniformly comparable to its continuous counterpart--the classical extremal distance. 
The topological rectangle $(\Omega;a,b,c,d)$ is crossed by $\oplus$ in an Ising configuration if there exists a path of $\oplus$ connecting $(ab)$ to $(cd)$ in $\Omega$. We denote this event by $\{(ab)\xlongleftrightarrow{\oplus}(cd)\}$. We will use the following strong RSW estimate. 

\begin{lemma}{\cite[Corollary~1.7]{ChelkakDuminilHonglerCrossingprobaFKIsing}} \label{lem::RSW}
	For each $L>0$, there exists $c(L)>0$ such that the following holds: for any discrete topological rectangle $(\Omega;a,b,c,d)$ with $\mathrm{d}_{\Omega}\left((ab),(cd)\right)\leq L$, we have
	\begin{equation*}
		\mu_{\beta_c,\Omega}^{\mathrm{mixed}} \left[(ab)\xlongrightarrow{\oplus}(cd)\right]\geq c(L),
	\end{equation*}
	where the boundary conditions are free on $(ab)\cup (cd)$ and $\ominus$ on $(bc)\cup (da)$. 
\end{lemma}

With Lemma~\ref{lem::RSW}, it is immediate to prove Lemma~\ref{lem::Ising_coupling}
\begin{proof}[Proof of Lemma~\ref{lem::Ising_coupling}]
	This can be proved in the same way as the proof of~\cite[Proposition~3.1]{GarbanPeteSchrammPivotalClusterInterfacePercolation} for percolation model, with RSW estimate for critical percolation model replaced by the RSW estimates Lemma~\ref{lem::RSW} for critical Ising model. 
	\end{proof}

\begin{lemma}\label{lem::spatial}
	There exist universal constants $c_5,c_6>0$ such that for any $0<r_1<r_2<1$, any boundary conditions $\tau,\xi$ such that $\tau=\xi$ on $\Gamma_{1}^{\delta}$, and any event $E$ depending on the status of vertices in $\Lambda_{r_2}^{\delta}$, we have the following estimate for all $\delta>0$ small enough: 
	\begin{align*}
		\left | 1-\frac{\mu_{\beta_c, \Lambda_{r_1}^{\delta}}^{\tau}[E]}{ \mu_{\beta_c, \Lambda_{r_1}^{\delta}}^{\xi}[E]}\right |\leq c_5\left(\frac{r_1}{r_2}\right)^{c_6}. 
	\end{align*}
\end{lemma}
\begin{proof}
	The proof is essentially the same as that of~\cite[Proposition~5.11]{DuminilCopinHonglerNolinRSWFKIsing}. 
\end{proof}

We now introduce arm events and well-separated arm events, which are important for the proof of Lemma~\ref{lem::compare_arm_events}. Fix $0<r_1<r_2\leq 1$ and consider the discrete annular domain $\Lambda_{r_1}^{\delta}\setminus \Lambda_{r_2}^{\delta}$. A simple path of $\oplus$ or of $\ominus$ connecting $\Gamma_{r_2}^{\delta}$ to $\Gamma_{r_1}^{\delta}$ is called an \textit{arm}. We define $\mathcal{C}_{\delta}(r_1,r_2)$ to be the event that there are $2N$ disjoint arms $(\gamma_j^{\delta})_{1\leq j\leq 2N}$ connecting $\Gamma_{r_2}^{\delta}$ to $\Gamma_{r_1}^{\delta}$ in the discrete annular domain $\Lambda_{r_1}^{\delta}\setminus \Lambda_{r_2}^{\delta}$ with alternating spins $\oplus/\ominus/\cdots/\oplus/\ominus$, where $(\gamma_{j}^{\delta})_{1\leq j\leq 2N}$ are in counterclockwise order.

A \textit{landing sequence} $(I_j)_{1\leq j\leq 2N}$ is a sequence of disjoint subintervals on $\partial\mathbb{D}$ in counterclockwise order. We denote by $z(I_j)$ the center of $I_j$. We say $(I_j)_{1\leq j\leq 2N}$ is \textit{$\epsilon$-separated} if:
\begin{enumerate}[label=\textnormal{(\arabic*):}, ref=\arabic*]
	\item the intervals are at distance at least $2\epsilon$ from each other;
	\item for each $I_j$, the length of $I_j$ is at least $2\epsilon$. 
\end{enumerate}
We let $\omega_j=\oplus$ if $j$ is odd and $\omega_j=\ominus$ if $j$ is even. Given two $\epsilon$-separated landing sequences $(I_j)_{1\leq j\leq 2N}$ and $(I'_j)_{1\leq j\leq 2N}$. We say that the arms $(\gamma_{k}^{\delta})_{1\leq j\leq 2N}$ are \textit{$\epsilon$-well-separated} with landing sequence $(I_j)_{1\leq j\leq 2N}$ on $\Gamma_{r_2}^{\delta}$ and landing sequence $(I'_j)_{1\leq j\leq 2N}$ on $\Gamma_{r_1}^{\delta}$ with alternating spins if:
\begin{enumerate}[label=\textnormal{(\arabic*):}, ref=\arabic*]
	\item \label{item::res_1} for each $j$, the arm $\gamma_j^{\delta}$  connects $r_2I_j$ to $r_1I'_j$; 
	\item \label{item::res_2} for each $j$, the arm $\gamma_j^{\delta}$ can be $\omega_j$-connected to distance $\epsilon r_1$ of $\Gamma_{r_1}^{\delta}$ inside 
	\[\{z=rz': z'\in I'_j, r\in (r_1-\epsilon r_1,r_1+\epsilon r_1)\}; \] 
	\item for each $j$, the arm $\gamma_j^{\delta}$ can be $\omega_j$-connected to distance $\epsilon r_2$ of $\Gamma_{r_2}^{\delta}$ inside 
	\[\{z=rz': z'\in I_j, r\in (r_2-\epsilon r_2,r_2+\epsilon r_2)\}.\]
\end{enumerate}
We denote this event by $\mathcal{C}_{\delta}^{I/I'}(r_1,r_2;\epsilon)$. 

\begin{lemma}\label{lem::Ising_well_separa}
	Fix $\epsilon>0$. There exist two constants $c_7,c_8>0$ that only depend on $\epsilon$ and $N$ so that for any two $\epsilon$-separated landing sequences $(I_j)_{1\leq j\leq 2N}$, $(I'_j)_{1\leq j\leq 2N}$, for any $0<2r_0<r_1<r_2<r_3/2$ and for any boundary condition $\tau$ on $\Gamma_{r_0}\cup \Gamma_{r_3}$, we have the following estimate for all $\delta>0$ small enough:
	\begin{equation*}
		c_7\leq \frac{		\mu^{\tau}_{\beta_c, \Lambda_{r_0}^{\delta}\setminus \Lambda_{r_3}^{\delta}}\left[\mathcal{C}_{\delta}(r_1,r_2)\right]}{		\mu^{\tau}_{\beta_c, \Lambda_{r_0}^{\delta}\setminus \Lambda_{r_3}^{\delta}}\left[\mathcal{C}^{I/I'}_{\delta}(r_1,r_2;\epsilon)\right]}\leq c_8.
	\end{equation*}
\end{lemma}
\begin{proof}
The proof is exactly the same as the proof of~\cite[Proposition~5.6]{ChelkakDuminilHonglerCrossingprobaFKIsing}. The key ingredient is the RSW estimates Lemma~\ref{lem::RSW}. 
\end{proof}

From now on, we fix a $\frac{1}{8N}$-separated landing sequence $(I_j)_{1\leq j\leq 2N}$ and write $\mathcal{C}_{\delta}^{I}(r_1,r_2):=\mathcal{C}_{\delta}^{I/I}(r_1,r_2;\frac{1}{8N})$. For fixed $r\in (0,1)$, we let $(R_j^r)_{1\leq j\leq 2N}$ be $2N$ disjoint topological rectangles (that is, polygons with four marked points) in $\{z: r<|z|<1\}$ so that for each $j$, 
\[R_j^r\cap \{z:r<|z|<(1+\frac{1}{8N})r\}=\{z=r'z': z'\in I_j, r'\in (r,(1+\frac{1}{8N})r)\}\]
and that there are $2$ sides of $R_j^r$ lying on $(x_j,x_{j+1})$ and $rI_j$, respectively. Now we are ready to prove Lemma~\ref{lem::compare_arm_events}.

\begin{proof}[Proof of Lemma~\ref{lem::compare_arm_events}]
We adopt the same notations as in the proof of~\eqref{eqn::IsingAnnulusCvg}.  We denote
\[E_p^{\delta}=\{\eta_1^{\delta},\ldots,\eta_{2N}^{\delta}\text{ all hit }\Gamma_{e^{-p}}^{\delta}\}.\]
According to the convergence of discrete interfaces under $\PP^{\delta}_{\LA}$ and $\PP^{\delta}_{\Omega}$, we have 
\begin{align}
			\lim_{\delta\to 0}\mathbb{P}^{\delta}_{\LA}[E_p^{\delta}]=\PP_{\mathrm{Ising}}^{(\mathbb{A}_p;\bs{x})}[E_p], \qquad
					\lim_{\delta\to 0}\PP^{\delta}_{\Omega}[E_p^{\delta}]=\PP_{\Ising}^{(\unitD;\bs{x})}[E_p]. \label{eqn::IsingAsy_aux5}
\end{align}
Now let us compare $\mathbb{P}^{\delta}_{\LA}[E_p^{\delta}]$ and $\PP_{\Omega}^{\delta}[E_p^{\delta}]$ in the discrete.

First, we choose $C>1$ so that for constants $c_5$ and $c_6$ in Lemma~\ref{lem::spatial}, we have 
\begin{equation}\label{eqn::IsingAsy_aux7}
	c_5(\frac{1}{C})^{c_6}\leq \frac{1}{2}.
\end{equation}
Without loss of generality, We may assume that $p$ is large enough so that $Ce^{-p}<\frac{1}{100}$.	From Lemma~\ref{lem::spatial}, when $\delta>0$ is small enough, we have 
\begin{equation}\label{eqn::IsingAsy_aux6}
	1/2\le 1-c_5C_p^{-c_6}\leq \frac{\PP^{\delta}_{\LA}[\eta_1^{\delta},\ldots,\eta_{2N}^{\delta}\text{ all hit }\Gamma_{Ce^{-p}}^{\delta}]}{\PP_{\Omega}^{\delta}[\eta_1^{\delta},\ldots,\eta_{2N}^{\delta}\text{ all hit }\Gamma_{Ce^{-p}}^{\delta}]}\leq 1+c_5C_p^{-c_6}\le 3/2.
\end{equation}

Second, we prove that there exists a constant $c_9>0$ that depends only on $N$ and $C$ in~\eqref{eqn::IsingAsy_aux7} so that 
\begin{equation}\label{eqn::IsingAsy_aux8}
	1\leq\frac{\PP^{\delta}_{\LA}[\eta_1^{\delta},\ldots,\eta_{2N}^{\delta}\text{ all hit }\Gamma_{Ce^{-p}}^{\delta}]}{\PP^{\delta}_{\LA}[\eta_1^{\delta},\ldots,\eta_{2N}^{\delta}\text{ all hit }\Gamma_{e^{-p}}^{\delta}]}\leq c_{9},
\end{equation}
\begin{equation}\label{eqn::IsingAsy_aux9}
	1\leq\frac{\PP^{\delta}_{\Omega}[\eta_1^{\delta},\ldots,\eta_{2N}^{\delta}\text{ all hit }\Gamma_{Ce^{-p}}^{\delta}]}{\PP^{\delta}_{\Omega}[\eta_1^{\delta},\ldots,\eta_{2N}^{\delta}\text{ all hit }\Gamma_{e^{-p}}^{\delta}]}\leq c_{9}.
\end{equation}
The left hand sides follow immediately from the fact that $C>1$. Now let us prove right hand sides. Note that 
\begin{align}
	\PP^{\delta}_{\LA}[\eta_1^{\delta},\ldots,\eta_{2N}^{\delta}\text{ all hit }\Gamma_{Ce^{-p}}^{\delta}]=&\PP^{\delta}_{\LA}[\mathcal{C}_{\delta}(Ce^{-p},1)]\notag\\
	\leq& \PP^{\delta}_{\LA}[\mathcal{C}_{\delta}(Ce^{-p},\frac{1}{10})]\notag\\
	\leq & c_8\PP^{\delta}_{\LA}[\mathcal{C}^{I}_{\delta}(Ce^{-p},\frac{1}{10})] , \label{eqn::IsingAsy_aux10}
\end{align}
where the last inequality is due to Lemma~\ref{lem::Ising_well_separa}.
Conditional on the event $\mathcal{C}^{I}_{\delta}(Ce^{-p},\frac{1}{10})$, for each $j$, denote by $\mathrm{Cross}_j^{(1)}$ the event that there exist $2$ paths with spin $\omega_j$ so that one path connects $\frac{1}{10}I_j$ to $\partial\mathbb{D}$ inside $R_j^{\frac{1}{10}}$ and that the other path connectes $\{z=rz_1(j): r\in (\frac{1}{10},\frac{1}{10}(1+\frac{1}{8N}))\}$ to $\{z=rz_2(j): r\in (\frac{1}{10},\frac{1}{10}(1+\frac{1}{8N}))\}$ inside
\[\{z=rz': z'\in I_j, r\in (\frac{1}{10},\frac{1}{10}(1+\frac{1}{8N}))\},\]
where $z_1(j),z_2(j)$ are two end points of $I_j$. Conditional on the event $\mathcal{C}^{I}_{\delta}(Ce^{-p},\frac{1}{10})$, for each $j$, denote by $\mathrm{Cross}_j^{(2)}$ the event that there exist $2$ paths with spin $\omega_j$ so that one path connects $Ce^{-p}I_k$ to $e^{-p}I_k$ inside $\{z=rz': z\in I_k, r\in (e^{-p},Ce^{-p})\}$ and that the other path connects $\{z=rz_1(j): r\in (Ce^{-p}(1-\frac{1}{8N}),Ce^{-p})\}$ to $\{z=rz_2(j): r\in (Ce^{-p}(1-\frac{1}{8N}),Ce^{-p})\}$ inside
\[\{z=rz': z'\in I_j, r\in (Ce^{-p}(1-\frac{1}{8N}),Ce^{-p})\}.\]
Note that
\begin{equation}\label{eqn::IsingAsy_aux11}
	\mathcal{C}^{I}_{\delta}(Ce^{-p},\frac{1}{10})\bigcap\left(\bigcap_{j=1}^{2N}\left(\mathrm{Cross}_j^{(1)}\cap \mathrm{Cross}_j^{(2)}\right)\right) \subseteq \mathcal{C}_{\delta}(e^{-p},1).
\end{equation}
Thanks to FKG inequality and RSW estimates for Ising model, we find a universal constant $c>0$ so that the conditional probability
\begin{equation}\label{eqn::IsingAsy_aux12}
	\PP^{\delta}_{\LA}\left[\bigcap_{j=1}^{2N}\left(\mathrm{Cross}_j^{(1)}\cap \mathrm{Cross}_j^{(2 )}\right)|\mathcal{C}^{I}_{\delta}(Ce^{-p},\frac{1}{10})\right]\geq c.
\end{equation}
Combining~\eqref{eqn::IsingAsy_aux10}-\eqref{eqn::IsingAsy_aux12}, we obtain the right hand side of~\eqref{eqn::IsingAsy_aux8}. Similarly, we can obtain the right hand side of~\eqref{eqn::IsingAsy_aux9}. Thus, we obatin~\eqref{eqn::IsingAsy_aux8} and~\eqref{eqn::IsingAsy_aux9}. 

Combining~\eqref{eqn::IsingAsy_aux6}-\eqref{eqn::IsingAsy_aux9} with~\eqref{eqn::IsingAsy_aux5}, we obtain~\eqref{eqn::IsingAsy_aux2} and complete the proof.
\end{proof}

%

{\small\newcommand{\etalchar}[1]{$^{#1}$}
}

\end{document}